\documentclass[a4paper, 12pt]{article}

\usepackage[sort&compress]{natbib}
\bibpunct{(}{)}{;}{a}{}{,} 

\usepackage{ifthen} 
\usepackage[usenames]{color}
\usepackage{fullpage}
\usepackage[dvipsnames]{xcolor}
\usepackage{subfigure}

\usepackage{graphicx} 
\usepackage{amsthm, amsmath, amssymb, mathrsfs, multirow, url}
\usepackage{amsfonts}

\usepackage{caption}

\usepackage{empheq}
\usepackage{cases}
   {\addtolength{\leftskip}{#1}\addtolength{\rightskip}{#2}}{\par}
\usepackage[dvipsnames]{xcolor}

\theoremstyle{definition}
\newtheorem{defn}{Definition}

\theoremstyle{plain} 
\newtheorem{thm}{Theorem}
\newtheorem{cor}{Corollary}
\newtheorem{prop}{Proposition}
\newtheorem{lem}{Lemma}

\theoremstyle{remark} 
\newtheorem{remark}{Remark}

\newcommand{\prob}{\mathsf{P}}

\newcommand{\unif}{{\sf Unif}}

\newcommand{\gam}{{\sf Gamma}}

\newcommand{\RR}{\mathbb{R}}

\newcommand{\YY}{\mathbb{Y}}

\newcommand{\J}{\mathcal{J}}
\newcommand{\I}{\mathcal{I}}

\newcommand{\veta}{\boldsymbol{\veta}}

\renewcommand{\phi}{\varphi} 
\newcommand{\eps}{\varepsilon}

\DeclareMathAlphabet      {\mathbfit}{OML}{cmm}{b}{it}

\title{On optimal designs for non-regular models}
\author{Yi Lin\footnote{Department of Mathematics, Statistics, and Computer Science, University of Illinois at Chicago. Email: {\tt ylin46@uic.edu}, {\tt myang2@uic.edu}} \qquad Ryan Martin\footnote{Department of Statistics, North Carolina State University. Email: {\tt rgmarti3@ncsu.edu}} \qquad Min Yang$^*$}
\date{\today}

\begin{document}

\maketitle

\begin{abstract}
Classically, Fisher information is the relevant object in defining optimal experimental designs.  However, for models that lack certain regularity, the Fisher information does not exist and, hence, there is no notion of design optimality available in the literature.  This article seeks to fill the gap by proposing a so-called {\em Hellinger information}, which generalizes Fisher information in the sense that the two measures agree in regular problems, but the former also exists for certain types of non-regular problems.  We derive a Hellinger information inequality, showing that Hellinger information defines a lower bound on the local minimax risk of estimators.  This provides a connection between features of the underlying model---in particular, the design---and the performance of estimators, motivating the use of this new Hellinger information 
for non-regular optimal design problems.  Hellinger optimal designs are derived for several non-regular regression problems, with numerical results empirically demonstrating the efficiency of these designs compared to alternatives.   

\smallskip

\emph{Keywords and phrases:} E-optimality; experimental design; Fisher information; Hellinger distance; information inequality. 
\end{abstract}

\section{Introduction}
\label{S:intro}  

 Optimal experimental design is a classical problem with substantial recent developments.  For example, \citet{biedermann2006optimal}, \citet{dette2008optimal}, \citet{feller2017optimal}, and \citet{schorning2017optimal} studied optimal designs for dose-response models; \citet{dette2016optimal} and \citet{dette2017new} investigated optimal designs for correlated observations; 
\citet{dror2006robust} and \citet{gotwalt2009fast}  studied robustness issues in optimal designs; \citet{lopez2007optimal}, \citet{waterhouse2008design}, \citet{biedermann2009constrained}, \citet{dette2009optimal}, and \citet{dette2018optimal} studied optimal discrimination designs; \citet{biedermann2011optimal} studied optimal design for additive partially nonlinear models; \citet{yu2011d}, \citet{yang.biedermann.tang.2013}, \citet{sagnol2015computing}, and \citet{harman2017barycentric} investigated algorithms for deriving optimal designs; and \citet{yang2009support}, \citet{yang2010garza}, \citet{dette2011note}, \citet{yang2012identifying}, and \citet{dette2013complete} built a new theoretical framework for studying optimal designs.  The focus of these developments has been exclusively on regular models that enjoy certain normal features asymptotically, 
such as generalized linear models.  However, certain non-regular models may be appropriate in practical applications \citep[e.g.,][]{chernozhukov2004likelihood, hirose1997inference, cousineau2009fitting}.  In particular, \citet{smith1994nonregular} describes a class of non-regular linear regression models, 
\[ y = x^\top \theta + \eps, \]
where the error $\eps$ is non-negative, which implies a non-regular model for $y$, given $x$, since its distribution has $\theta$-dependent support. Such models are useful if the goal is to study extremes; for example, $x^\top \theta$ might represent the lower bound on remission time when a patient is subjected to treatment settings described by the vector $x$. To date, there is no literature on optimal designs for cases like this, and the goal of this paper is to fill this gap by developing an approach to optimal design in non-regular problems.  

Towards formulating a design problem in a non-regular model, the first obstacle is that the Fisher information matrix---the fundamental object in the classical optimal design context---does not exist.  To overcome this, we draw inspiration from recent work on the development of non-informative priors in the Bayesian context, thereby backtracking the path taken by \citet{lindley1956} and \citet{bernardo1979} from information in an experiment to non-informative priors. In particular, \citet{shemyakin2014hellinger} proposes an alternative to Fisher information and generalizes the non-informative prior construction of Jeffreys.  An important feature of the Fisher information is how it describes the local behavior of the Hellinger distance (see Section~\ref{S:review}), leading to its connection to estimator quality via the information inequality.  Unfortunately, the role that Shemyakin's information 
plays in the local approximation of Hellinger distance for multi-parameter models remains unclear; 
see Remark~\ref{re:shemyakin}. 
Since a connection to the quality of estimators is essential to our efforts to define a meaningful notion of optimal designs, we take an alternative approach where the focus is on a local approximation of Hellinger distance.



We start by looking at the local behavior of the squared Hellinger distance between models $P_\theta$ and $P_\vartheta$, for $\vartheta \approx \theta$.  In the regular cases, there is a local quadratic approximation to the squared distance and the Fisher information matrix appears in the approximating quadratic form.  In non-regular problems, by definition, the squared Hellinger distance is not locally quadratic, so there is no reason to expect that an ``information matrix'' can be extracted from this approximation.  In fact, not being differentiable in quadratic mean implies that the Hellinger distance is continuous at $\theta$, but not differentiable, so important features of the local approximation will generally depend on both the magnitude and the direction of the departure of $\vartheta$ from $\theta$.  From the local Hellinger distance approximation for a given direction, we define a direction-dependent {\em Hellinger information}, which is additive like Fisher information for independent data, and establish a corresponding information inequality that suitably lower-bounds the risk function of an arbitrary estimator along that direction.  The direction-dependence is removed via profiling, and the result is a locally minimax lower bound on the risk of arbitrary estimators, which is inversely related to our direction-free Hellinger information. Therefore, just like in the familiar Cram\'er--Rao inequality for regular models, larger Hellinger information means a smaller lower bound and, consequently, better estimation in terms of risk.  

The established connection between our Hellinger information for non-regular models and the quality of estimators provides a natural path to approach the optimal design problem.  In particular, our Hellinger information depends on the design, so we define the optimal design as one that maximizes the Hellinger information.  The intuition, just like in the regular case, is that maximizing the information minimizes the lower bound on the risk, thereby leading to improved estimation.  If the model happens to be regular, then our proposed optimal design corresponds to the classical E-optimal design that maximizes the minimum eigenvalue of the Fisher information matrix, so the new approach at least has intuitive appeal.  After formally defining the notion of optimal design in this context, we develop some novel theoretical results, in particular a complete class theorem for symmetric designs in the context of non-regular polynomial regression.  This theorem, along with some special cases presented in Propositions~\ref{prop:linear}--\ref{prop:quadratic}, suggests the potential for a line of developments parallel to that for regular models.  

The remainder of the paper is organized as follows.  Section~\ref{S:review} sets our notation and briefly reviews the Fisher information and its properties under regularity conditions.  We relax those regularity conditions in Section~\ref{S:info} and develop a notion of Hellinger information for certain non-regular models.  The main result of the paper, Theorem~\ref{thm:bound}, establishes a connection between this Hellinger information and the quality of estimators, thus paving the way for a framework of optimal designs for non-regular models in Section~\ref{S:design}.  Some specific non-regular regression models are considered in Section~\ref{S:special}, and we derive some analytical optimality results and some numerical demonstrations of the improved efficiency of the optimal designs over other designs.  Some concluding remarks are given in Section~\ref{S:discuss} and proofs of the two main theorems are presented in Appendix~\ref{S:proofs1}; the remaining details are given in the Supplementary Material \citep{nonreg.supp}.

\section{Review of information in regular models}
\label{S:review}

The proposed model assumes that the $\YY$-valued observations $Y_1,\ldots,Y_n$ are independent, and the marginal distribution of $Y_i$ is $P_{i,\theta}$, where $\theta$ is a fixed and unknown parameter in $\Theta \subseteq \RR^d$.  For example, $P_{i,\theta}$ might be a distribution that depends on both the parameter $\theta$ and a fixed covariate vector $x_i$.  We will further assume that, for each $i=1,\ldots,n$, $P_{i,\theta}$ has a density $p_{i,\theta}$ with respect to a fixed dominating $\sigma$-finite measure $\mu$.  When the index $i$ is not important, and there is no risk of confusion, we will drop the index and write simply $p_\theta$ for the density function with respect to $\mu$.  

It is common to assume that the model is regular in the sense that $\theta \mapsto p_{i,\theta}(y)$ is smooth for each $y$, and that $\theta$-derivatives of expectations can be evaluated by interchanging differentiation and integration.  For example, under conditions (6.6) in \citet{lehmann.casella.book}, one can define the $d \times d$ Fisher information matrix $I_i(\theta)$, whose $(k,\ell)$ entry is given by 
\begin{equation}
\label{eq:fim}
E_\theta\Bigl\{\frac{\partial}{\partial\theta_k} \log p_{i,\theta}(Y_i) \cdot \frac{\partial}{\partial\theta_\ell} \log p_{i,\theta}(Y_i) \Bigr\}, \quad k,\ell = 1,\ldots,d. 
\end{equation}
The Fisher information matrix can be defined in broader generality for families of distributions with a {\em differentiability in quadratic mean} property \citep[e.g.,][]{pollard1997another, van1998asymptotic}.  That is, assume that there exists a function $\dot\ell_\theta$, typically the gradient of $\log p_\theta$, taking values in $\RR^d$, such that 
\[ \int \bigl( p_{\theta+\eps}^{1/2} - p_\theta^{1/2} - \tfrac12 \eps^\top \dot\ell_\theta p_\theta^{1/2} \bigr)^2 \,d\mu = o(\|\eps\|^2), \quad \eps \to 0, \]
where $\|\cdot\|$ denotes the $\ell_2$-norm.  Then the Fisher information matrix exists and is given by the formula $I(\theta) = \int \dot\ell_\theta \dot\ell_\theta^\top \, p_\theta \,d\mu$.  If we let $H(P_\theta,P_\vartheta)$ denote the Hellinger distance and define $h$ as 
\[ h(\theta; \vartheta) \equiv H^2(P_\theta, P_\vartheta) := \int (p_\theta^{1/2} - p_\vartheta^{1/2})^2 \,d\mu = 2 - 2 \int (p_\theta p_\vartheta)^{1/2} \,d\mu, \]
then the above condition amounts to $h$ being locally quadratic:
\[ h(\theta; \theta+\eps) = \tfrac14 \, \eps^\top I(\theta) \eps + o(\|\eps\|^2). \]
Therefore, a model is regular if the squared Hellinger distance is locally approximately quadratic, with the Fisher information matrix characterizing that quadratic approximation.  This is the description of Fisher information that we will attempt to extend to the non-regular case below.  

Recall, also, that Fisher information is additive under independence.  That is, if $Y_1,\ldots,Y_n$ are independent, with $Y_i \sim p_{i,\theta}$, regular as above for each $i$, then the Fisher information in the sample of size $n$ satisfies 
\[ \I_n(\theta) = \sum_{i=1}^n I_i(\theta), \]
where $I_i(\theta)$ is the Fisher information matrix in \eqref{eq:fim} based on $p_{i,\theta}$ alone.  This property has a nice interpretation: larger samples have more information.  

Under differentiability in quadratic mean, one can prove an {\em information inequality} which states that, for any unbiased estimator $T=T(Y_1,\ldots,Y_n)$ of $m(\theta) = E_\theta(T) \in \RR$ with finite second moment, the variance is lower-bounded and satisfies
\[ V_\theta(T) \geq \dot m(\theta)^\top \, \I_n(\theta)^{-1} \, \dot m(\theta), \]
where $\dot m(\theta)$ is the gradient of $m(\theta)$; see \citet{pollard.chap6}.  The information inequality above, and its various extensions, establishes a fundamental connection between the quality of an estimator---in this case, the variance of an unbiased estimator---and the Fisher information matrix.  This connection has been essential to the development of optimal design theory and practice since the quality of an estimator can be ``optimized'' by choosing a design that makes the quadratic form in the lower bound as small as possible, or the Fisher information as large as possible.    


Finally, differentiability in quadratic mean implies local asymptotic normality \citep[e.g.,][Theorem~7.2]{van1998asymptotic} which is almost all one needs to show that maximum likelihood estimators are efficient in the sense that they attain the information inequality lower bound \citep[e.g.,][Theorem~7.12]{van1998asymptotic}.  Therefore, in sufficiently regular problems, there is a general procedure for constructing high-quality estimators, and that the quality of such estimators is controlled by the Fisher information matrix.  The remainder of this paper is concerned with non-regular cases
and, unfortunately, these differ from their regular counterparts in several fundamental ways.  First, the Fisher information is not well-defined in non-regular cases, so we have no general way of measuring the quality of estimators.  Second, one cannot rely on maximum likelihood for constructing good estimators.  For example, Le~Cam writes \citep[see][p.~674]{vaart.lecam.2002}

\begin{quote}
The author is firmly convinced that a recourse to maximum likelihood is justifiable only when one is dealing with families of distributions that are extremely regular. The cases in which maximum likelihood estimates are readily obtainable and have been proved to have good properties are extremely restricted.
\end{quote}
Therefore, to achieve our goals, we need a measure of information that is flexible enough to handle non-regular problems and is connected to estimation quality in general, but does not depend on a particular estimator.  The {\em Hellinger information}, defined in Section~\ref{SS:hellinger}, will meet these criteria and will provide a basis for defining optimal designs in non-regular problems. 


\section{Information in non-regular models}
\label{S:info}

\subsection{Definition and basic properties}
\label{SS:hellinger}

To start, we consider the scalar case with $d=1$.  Suppose that there exists a constant $\alpha \in (0,2]$ such that, for each $\theta$, the limit $J(\theta) = \lim_{\eps \to 0} |\eps|^{-\alpha} h(\theta; \theta + \eps)$ exists, is finite, and non-zero.  If such an $\alpha$ exists, then it must be unique; but there are cases where existence fails, e.g., when $\theta$ is not identifiable, so that $h(\theta, \theta+\eps) \equiv 0$ for all sufficiently small $\eps$.  The case $\alpha=2$ corresponds to differentiable in quadratic mean and, hence, ``regular,'' while $\alpha \in (0,2)$ corresponds to ``non-regular.'' Differentiability of $\vartheta \mapsto H(P_\theta, P_\vartheta)$ or lack thereof determines a model's regularity, so the largest value $\alpha$ can take is 2; otherwise, the limit is infinite. From the above limit, there is a local approximation, 
\begin{equation}
\label{eq:scalar.approx}
h(\theta;\vartheta) = J(\theta) |\theta-\vartheta|^\alpha + o(|\theta-\vartheta|^\alpha).
\end{equation}
This resembles the local H\"older condition considered in \citet[][Section~I.6]{ibragimov1981statistical}.  We call $\alpha$ the {\em index of regularity} and $J(\theta)$ the {\em Hellinger information}.  Of course, if $\alpha=2$, then $J(\theta)$ is proportional to $I(\theta)$, the Fisher information.  Next are a few quick examples, all with $\alpha=1$. 
\begin{itemize}
\item If $P_\theta=\unif(0,\theta)$, $\theta > 0$, then $J(\theta) = \theta^{-1}$.
\item If $P_\theta=\unif(\theta^{-1}, \theta)$, $\theta > 1$, then $J(\theta) = (\theta^2 + 1) \{\theta (\theta^2 - 1)\}^{-1}$. 
\item If $P_\theta=\unif(\theta; \theta^2)$, $\theta > 1$, then $J(\theta) = (2\theta + 1) \{\theta (\theta - 1)\}^{-1}$.
\end{itemize}

A class of non-regular models of particular interest to us here are those in \citet{smith1994nonregular} based on location shifts of distributions supported on the positive half-line.  Consider a density $p_0$ on $(0,\infty)$ that satisfies
\begin{equation}
\label{eq:p0}
p_0(y) = \beta \, c \, y^{\beta - 1}, \quad \text{as $y \to 0$,}
\end{equation}
where $\beta \geq 1$ and $c=c(\beta) \in (0,\infty)$.  For example, the gamma and Weibull families, with shape parameter $\beta$ and scale $\sigma$, have $c = \{\beta \sigma^\beta \Gamma(\beta)\}^{-1}$ and $c=\sigma^{-\beta}$, respectively.  The next result identifies the regularity index $\alpha$ and the Hellinger information $J(\theta)$ for this class of location parameter problems, with $p_\theta(y) = p_0(y-\theta)$. It shows that $\alpha$ need not be an integer and the Hellinger information, like Fisher's, is constant in location models. When $\beta \geq 2$, the model is regular---with $\alpha=2$ and the Fisher information defined as usual---so we focus here on the non-regular case with $\beta \in [1,2)$.  

\begin{prop}
\label{J.shortcut}
Let $p_0$ satisfy \eqref{eq:p0} with $\beta \in [1,2)$.  If, for some $\Delta > 0$, 
\begin{equation}
\label{eq:right.tail}
\int_\Delta^\infty \Bigl( \frac{d}{dy} \log p_0(y) \Bigr)^2 p_0(y) \,dy < \infty,
\end{equation}
then $\alpha = \beta$ and $J(\theta) \equiv c\{1 + \beta \, r(\beta)\}$, where $c$ is as in \eqref{eq:p0} and  
\begin{equation}
\label{eq:r.beta}
r(\beta) = \int_0^\infty \{(w+1)^{(\beta-1)/2} - w^{(\beta-1)/2}\}^2 \,dw. 
\end{equation}
\end{prop}
 
\begin{proof}
See Section S2.1 in the Supplementary Material.
\end{proof}

\citet[][Theorem~VI.1.1]{ibragimov1981statistical} show that, in this case, $h(\theta;\theta + \eps)= O(|\eps|^\beta)$ as $\eps \to 0$, but they do not identify $J(\theta)$.  
Similar results have appeared elsewhere in the literature on non-regular models; our condition \eqref{eq:right.tail} is basically the same as Condition~$C_5$ in \citet{woodroofe1974}, which is basically the same as Assumption~9 in \citet{smith1985}.  

Turning to the general, non-regular multi-parameter case, where $\Theta$ is an open subset of $\RR^d$, defining Hellinger information requires some additional effort. 
In particular, non-regularity implies that the familiar local quadratic approximation of $h$ fails, so we should not expect to have an ``information matrix'' to describe the local behavior in such cases.   In fact, $h(\theta; \vartheta)$ depends locally on the {\em direction} along which $\vartheta$ approaches $\theta$, so there is no ``direction-free'' summary of the local structure and, hence, no ``information matrix''; see Remark~\ref{re:shemyakin}.  But this lack of a convenient quadratic approximation need not stop us from defining a suitable {\em Hellinger information}.  

\begin{defn}
\label{def:info}
Let $\Theta$ be an open subset of $\RR^d$, for $d \geq 1$, and let $u$ denote a generic direction, a $d$-vector with $\|u\|=1$.  Suppose there exists $\alpha \in (0,2]$ such that, for all $\theta \in \Theta$ and all directions $u$, the following limit exists and is neither 0 nor $\infty$:
\begin{equation}\label{Def.J}
 \lim_{\eps \to 0} \frac{ h(\theta; \theta + \eps u) }{|\eps|^\alpha}= J(\theta; u). 
 \end{equation}
Then, the following local approximation holds: 
\begin{equation}
\label{eq:local.expansion}
h(\theta; \theta +\eps u) = J(\theta; u) \, |\eps|^\alpha + o(|\eps|^\alpha), \quad \eps \to 0. 
\end{equation}
This defines the {\em index of regularity} $\alpha$ and the {\em Hellinger information} $J(\theta; u)$ at $\theta$ in the direction of $u$. 
\end{defn}

Since the approximation \eqref{eq:local.expansion} is in terms of $|\eps|$, it follows that $J(\theta; u) = J(\theta; -u)$, so $J(\theta;u)$ really only depends on the {\em line} defined by $u$.  If $d=1$, then there is only one line, i.e., $u=\pm 1$, hence, for the scalar case, we can drop the $u$ argument entirely and write $J(\theta)$ as described above.  It is also worth pointing out that Definition~\ref{def:info} assumes that a single index $\alpha$ suffices to describe the regularity of a model with a $d$-dimensional parameter.  This is appropriate for the kinds of regression models we have in mind here, but can be a limitation in other cases; see Remark~\ref{re:one.alpha} below.  

As a quick example, let $P_\theta = \unif(\theta_1,\theta_1 + \theta_2)$, where $\theta_1 \in \RR$ and $\theta_2 > 0$.  In this form, $\theta_1$ and $\theta_2$ are location and scale parameters, respectively.  If $u=(u_1,u_2)$ is a generic vector on the unit circle, then $J(\theta; u) = \theta_2^{-1} g(u)$,
\label{twopar.unif} 
where $g(u)$ has a form which is slightly too complicated to present here; see Section~S1 in the Supplementary Material.  This expression agrees with the familiar properties of Fisher information for location--scale models. 


Although we do not define an ``information matrix'' in the non-regular case (see Remark~\ref{re:shemyakin}), when the model is regular, i.e., when $\alpha=2$, there are still some connections between our Hellinger information and the familiar Fisher information.  In particular, $J(\theta; u)$ is a quadratic form involving the Fisher information $I(\theta)$ and the direction $u$.  This gives an alternative explanation of how the regular models admit a separation of the dependence on $\theta$ and on the direction $u$ of departure from $\theta$.  

\begin{prop}
\label{prop:JequalsI}
For a regular model, with $\alpha=2$, if $I(\theta)$ denotes the $d \times d$ Fisher information matrix, then $J(\theta; u) = \tfrac14 u^\top \, I(\theta) \, u$.
\end{prop}

Another useful and familiar feature of Fisher information that also holds for Hellinger information is the reparametrization formula (Proposition~\ref{prop:repar}), which comes in handy for regression problems where the natural parameter is expressed as a function of covariates and another parameter.

\subsection{Hellinger information inequality}
\label{SS:inequality}

We now return to our original setup where $Y_1,\ldots,Y_n$ are independent, but not necessarily identically distributed, with $Y_i \sim P_{i,\theta}$, $i=1,\ldots,n$, and $\theta$ is an unknown parameter taking values in an open subset $\Theta$ of $\RR^d$ for some $d \geq 1$.    Let $P_\theta^n$ denote the joint distribution of $Y^n=(Y_1,\ldots,Y_n)$.  Motivated by the regression problems below, we assume that each $P_{i,\theta}$ has the same index of regularity, $\alpha \in (0,2]$. Following our intuition from the regular case, define the Hellinger information at $\theta$, in the direction of $u$, based on the sample of size $n$, as 
\begin{equation}
\label{eq:total}
\J_n(\theta; u) = \sum_{i=1}^n J_i(\theta; u). 
\end{equation}
where $J_i(\theta; u)$ is the Hellinger information based on $P_{i,\theta}$ as described above.  See Remark~\ref{re:additive} for more on this additivity property. Theorem~\ref{thm:bound} below will establish a suitable connection between $\J_n(\theta; u)$ and the quality of an estimator, and this will provide the necessary foundation for defining optimal designs for non-regular models.  

Suppose the goal is to estimate $\psi(\theta)$, where $\psi: \RR^d \to \RR^q$, $q \leq d$, is sufficiently smooth.  Let $T_n = T(Y^n)$ be an estimator of $\psi(\theta)$, and measure its quality by the risk
\begin{equation}
\label{eq:risk}
R_\psi(T_n, \theta) = E_\theta^n\|T_n - \psi(\theta)\|^2, 
\end{equation}
the $q$-vector version of mean square error, where expectation, $E_\theta^n$, is with respect to $P_\theta^n$.  This covers the case where $\psi(\theta)=\theta$ and $q=d$, so that interest is in the full parameter $\theta$, and the case where $\psi(\theta)$ is a single component of $\theta$ and $q=1$, as well as other intermediate cases.  Next is the aforementioned lower bound on the risk in terms of the total Hellinger information.

\begin{thm}
\label{thm:bound}
Let $Y^n=(Y_1,\ldots,Y_n)$ consist of independent observations with $Y_i \sim P_{i,\theta}$, $i=1,\ldots,n$.  Let $\alpha \in (0,2]$ denote the common index of regularity, and $\J_n(\theta; u)$ the corresponding Hellinger information in \eqref{eq:total}.  Let $\psi: \Theta \to \RR^q$ be a differentiable function with full-rank $q \times d$ derivative matrix $D_\psi(\theta)$, and let $T_n = T(Y^n)$ be any estimator of $\psi(\theta)$ with risk function defined in \eqref{eq:risk}. If $\eps_{n,u} = \{3 \J_n(\theta; u)\}^{-1/\alpha}$, and
\begin{equation}
\label{eq:accumulates}
\lim_{n \to \infty}\inf_u [ n^{-1} \J_n(\theta; u)] > 0,
\end{equation}
then, for all large $n$, 
\begin{equation}
\label{eq:minimax}
\inf_{T_n} \sup_{\vartheta \in B_n(\theta)} R_\psi(T_n, \vartheta) \gtrsim \Bigl[ \inf_u  \bigl\{ \|D_\psi(\theta) \, u\|^{-\alpha} \, \J_n(\theta; u) \bigr\} \Bigr]^{-2/\alpha}, 
\end{equation}
where $B_n(\theta) \subset \Theta$ is the region whose boundary is determined by the union of $\{\theta + \eps_{n,u} u\}$ over all directions $u$.
\end{thm}

\begin{proof}
See Appendix~\ref{SS:thm1}.  
\end{proof}

Two very brief comments: first, the universal constant hidden in ``$\gtrsim$'' is known and given in the proof; second, there is nothing special about ``3'' in the definition of $\eps_{n,u}$, any number strictly greater than 2 would suffice.

Some additional comments about the interpretation of Theorem~\ref{thm:bound} are in order.  First, the reason for taking supremum over a small ``neighborhood" of $\theta$ is that a lucky choice of $T_n \equiv \psi(\theta)$ has excellent performance at $\theta$, but poor performance at a nearby $\vartheta$. The theorem basically says that, if one looks at a locally uniform measure of risk, which prevents  ``cheating" towards or luck at a particular $\theta$, then one cannot have smaller risk than that in the lower bound \eqref{eq:minimax}.  The classical Cram\'er--Rao lower bound uses unbiasedness of the estimator to prevent this kind of cheating/luck.

To assess the sharpness of the bound in \eqref{eq:minimax} when regularity conditions do not apply, consider the case where $q=1$, so that $\psi(\theta)$ is a scalar function.  For the rate, if we consider the identically independently distributed case, so that $\J_n(\theta; u) = n J_1(\theta; u)$, then it follows that the lower bound is of order $n^{-2/\alpha}$, which agrees with the known minimax rate for estimators in non-regular models \citep[][Sec.~I.5]{ibragimov1981statistical}. Therefore, the bound cannot be improved in terms of dependence on the sample size. To assess the quality of the lower bound in terms of its dependence on $\theta$, if the observations come from $\unif(0,\theta)$, which has $\alpha=1$ and $J(\theta)=\theta^{-1}$, the maximum likelihood estimator is the sample maximum, and its mean square error is given by 
\[ \frac{\theta^2 n}{(n+1)^2(n+2)} + \Bigl(\frac{\theta n}{n+1} - \theta\Bigr)^2. \]
Asymptotically, this expression is proportional to $\theta^2 n^{-2}$, which agrees with our lower bound. Therefore, up to universal constants, the bound in Theorem~\ref{thm:bound} is sharp.  Whether there exists an estimator that can attain the bound exactly or asymptotically is unclear in general; see Remark~\ref{re:attain}. 

It is worth stating the special case where $\alpha=2$ as a corollary to Theorem~\ref{thm:bound}.  This reveals some connection to the classical Cram\'er--Rao bound, even though we do not have access to an information matrix, and demonstrates the generality of our result.  

\begin{cor}
\label{cor:bound}
When $\alpha=2$, if $\psi: \Theta \to \RR^q$ has $q \times d$ derivative matrix $D_\psi(\theta)$ of rank $q \leq d$, and $\I_n(\theta)$ is the positive definite $d \times d$ Fisher information matrix, then the lower bound in \eqref{eq:minimax} is proportional to 
\[ \lambda_{\max}\{ D_\psi(\theta) \I_n(\theta)^{-1} D_\psi(\theta)^\top \}, \]
where $\lambda_{\max}(A)$ denotes the maximal eigenvalue of a matrix $A$.  
\end{cor}

\begin{proof}
See Section~S2.2 in the Supplementary Material. 
\end{proof}

For comparison to the classical setting, if we take $\psi(\theta) = \theta$, then the expression in the above display simplifies to 
\begin{equation}
\label{eq:special.bound}
\lambda_{\max}\{\I_n(\theta)^{-1}\} = \lambda_{\min}^{-1}\{\I_n(\theta)\}. 
\end{equation}
Wanting the information matrix to have a large minimal eigenvalue is a familiar concept in the classical optimal design theory; see Section~\ref{S:design}.  

This and the previous subsection, along with the remarks in Section~\ref{SS:remarks}, establish some important properties and insights concerning our proposed Hellinger information.  A difficulty that has not yet been addressed is the dependence of $\J_n(\theta; u)$ on the arbitrary direction $u$.  However, the lower bound in \eqref{eq:minimax} is free of a direction, so it makes sense to formulate a {\em direction-free} Hellinger information based on that.  For a non-regular model as formulated above, with index of regularity $\alpha \in (0,2]$, we set the direction-free Hellinger information at $\theta$, for interest parameter $\psi(\theta)$, as 
\begin{equation}
\label{eq:J.without.u}
\J_n^\psi(\theta) = \inf_u \bigl\{ \|D_\psi(\theta) \, u\|^{-\alpha} \J_n(\theta; u) \bigr\}.
\end{equation}
In the special case where $\psi(\theta) = \theta$, this simplifies to 
\begin{equation}
\label{eq:J.without.u.special}
\J_n(\theta) = \inf_u \J_n(\theta; u). 
\end{equation}
Moreover, in the regular case with $\alpha=2$, it follows from Corollary~\ref{cor:bound} and, in particular, \eqref{eq:special.bound}, that $\J_n(\theta)$ above is (proportional to) the smallest eigenvalue of the Fisher information matrix.  Therefore, definition \eqref{eq:J.without.u} seems very reasonable; more details are presented in Section~\ref{S:design}.  

\subsection{Technical remarks}
\label{SS:remarks}

\begin{remark}
\label{re:one.alpha}
Definition~\ref{def:info} does not allow $\alpha$ to depend on $u$, so each component of $\theta$, treated individually, must have the same index of regularity.  To see this, consider an exponential distribution with location and rate parameters $\theta_1$ and $\theta_2$, respectively.  If $\theta_1$ was fixed and only $\theta_2$ was unknown, then it is a regular problem and the above definition would hold with $\alpha=2$.  Similarly, if $\theta_2$ was fixed and only $\theta_1$ was unknown, then the definition holds with $\alpha=1$ according to Proposition~\ref{J.shortcut}.  However, if both $\theta_1$ and $\theta_2$ are unknown, then the model does not satisfy the conditions of Definition~\ref{def:info}.  Consider two unit vectors $u=(1,0)$ and $u'=(0,1)$.  If $\alpha=1$, then $J(\theta; u)$ is in $(0,\infty)$ but $J(\theta;u')$ is zero; likewise, if $\alpha=2$, then $J(\theta;u')$ is in $(0,\infty)$ but $J(\theta; u)$ is infinite.  Therefore, the above definition cannot accommodate situations where the components of $\theta$, treated individually, would have different regularity indices.  But the design applications we have in mind in this paper fit naturally within a setting where all components have the same regularity; the more general case will be considered elsewhere.  
\end{remark}

\begin{remark}
\label{re:shemyakin}
Our definition of Hellinger information coincides with that in \citet{shemyakin2014hellinger} for one-parameter models, but our perspectives differ when it comes to multi-parameter models.  Shemyakin defines a ``Hellinger information matrix'' for non-regular problems, which seems to contradict our above claim that no such matrix is available, so some more detailed comments are necessary.  Shemyakin makes no claim that his information matrix is related to the local behavior of $h$, and we are unable to conclude definitively whether it is or is not.  We do know, however, that $\vartheta \mapsto h(\theta,\vartheta)$ is ``bowl-shaped'' (though not smooth) at each $\theta$, so if such a matrix could describe the local behavior, then it ought to be non-negative definite.  However, \citet[][p.~931]{shemyakin2014hellinger} admits that a general non-negative definiteness result has not been established for his information matrix.  Without a non-negative definiteness result for his Hellinger information matrix, lower bounds like those in, e.g., \citet{shemyakin1991, shemyakin1992} may not be informative, and its use in defining optimal designs lacks justification. 
\end{remark}

\begin{remark}
\label{re:additive}
In \eqref{eq:total} we {\em defined} the Hellinger information in an independent sample of size $n$ as $\J_n(\theta; u) = \sum_{i=1}^n J_i(\theta; u)$, the sum of the individual Hellinger information measures.  This, however, is not a choice made by us, it is a consequence of the proof of Theorem~\ref{thm:bound}.  To see this, heuristically, start with the Hellinger distance between joint distributions $P_\theta^n$ and $P_\vartheta^n$, assuming independence.  A straightforward calculation reveals 
\begin{align*}
H^2(P_\theta^n, P_\vartheta^n) 
& = 2 - 2 \prod_{i=1}^n \int \{ p_{i,\theta}(y_i) p_{i,\vartheta}(y_i) \}^{1/2} \, dy_i \\
& = 2 - 2 \exp\Bigl\{ \sum_{i=1}^n \log \bigl[ 1 - \tfrac12 H^2(P_{i,\theta}, P_{i,\vartheta}) \bigr] \Bigl\}
\end{align*}
Since $\log(1+x) \approx x$ for $x \approx 0$, if $\vartheta$ is sufficiently close to $\theta$, then the exponent is approximately $-\frac12 \sum_{i=1}^n H^2(P_{i,\theta}, P_{i,\vartheta})$ and then, by Taylor's theorem applied to $x \mapsto e^{-x}$ at $x \approx 0$, we conclude that 
\[ H^2(P_\theta^n, P_\vartheta^n) \approx \sum_{i=1}^n H^2(P_{i,\theta}, P_{i,\vartheta}). \]
Therefore, a local approximation of the left-hand side is roughly equal to a sum of local approximations on the right-hand side, which leads to \eqref{eq:total}.  
\end{remark}

\begin{remark}
\label{re:attain}
An important unanswered question in the above theory is if there are any estimators that are efficient in the sense that they attain the lower bound in Theorem~\ref{thm:bound} in some generality.  In the simple $\unif(0,\theta)$ example above, we showed that the bound is asymptotically attained, up to universal constants, by the sample maximum.  General results about the rate of convergence in non-regular models are consistent with our lower bound, but, to our knowledge, more precise results concerning the asymptotic behavior of estimators in non-regular problems is limited to certain special cases.  Our work here provides some motivation for further investigation of these asymptotic properties.  Not having an estimator that provably attains the lower bound complicates our attempts to demonstrate the efficiency gains of our proposed optimal designs in Section~\ref{S:design}, but a quality estimator is available in the applications we have in mind; see Section~\ref{SS:numerical}.  
\end{remark}

\section{Optimal designs for non-regular models}
\label{S:design}

\subsection{Definition}

The previous section built up a framework of information, based on a local approximation of the squared Hellinger distance, suitable for non-regular problems where Fisher information does not exist.  Our motivation for building such a framework was to address the problem of optimal experimental designs in cases where the underlying statistical model is non-regular.  This section defines what we mean by an optimal design for non-regular models, and provides some additional details about the Hellinger information that are particularly relevant to the design problem.  

We start here with a slightly different setup than in the previous section, but quickly connect it back to the preceding.  Let $Y_1,\ldots,Y_n$ be independent observations, where $Y_i$ has density function $q_{\eta_i}$, for $i=1,\ldots,n$.  That is, each $Y_i$ has its own parameter $\eta_i$, which we will assume is real-valued, as is typical in linear and generalized linear models.  Then the design problem proceeds by expressing the unit-specific parameter $\eta_i$ as a given function $g(x_i,\theta)$ of a common parameter $\theta \in \RR^d$ and a vector of unit-specific covariates; here, of course, the covariates are constants that the investigator is able to set in any way he/she pleases, but preferably in a way that is ``optimal'' in some sense.  By linking each $\eta_i$ to a common $\theta$, we obtain the setup from previous sections, i.e., $Y_i \sim p_{i,\theta}$, independent, for $i=1,\ldots,n$.  

The next result, stated in the context of $n=1$, parallels a familiar one in the regular case for Fisher information.  It aids in computing the Hellinger information under a reparametrization like the one described above. 

\begin{prop}
\label{prop:repar}
Let $q_\eta$ be a density function depending on a scalar parameter $\eta$, and suppose that the index of regularity is $\alpha \in (0,2]$ and the Hellinger information is $\tilde J(\eta)$.  Define a new density $p_\theta$, for $\theta \in \Theta \subseteq \RR^d$, as $q_{g(\theta)}$ where $g: \Theta \to \RR$ is a smooth function with non-vanishing gradient $\dot g$.  Then $p_\theta$ also has index of regularity $\alpha$, and the corresponding Hellinger information at $\theta$, in the direction of $u$, is 
\[ J(\theta; u) = |\dot g(\theta)^\top u|^\alpha \tilde J(g(\theta)). \]
\end{prop}

\begin{proof}
See Section~S2.3 in the Supplementary Material.
\end{proof}

From the general theory in Section~\ref{S:info}, if $Y_1,\ldots,Y_n$ are independent, then under the assumptions in Proposition~\ref{prop:repar}, i.e., $Y_i \sim p_{i,\theta} = q_{g_i(\theta)}$, the Hellinger information at $\theta$, in direction of $u$, is 
\[ \J_n(\theta; u) = \sum_{i=1}^n |\dot g_i(\theta)^\top u|^\alpha \tilde J(g_i(\theta)). \]

For the special case where $g_i(\theta) = g(x_i, \theta)$ for covariates $x_i$, it is clear that $\J_n(\theta; u)$ depends on $x_1,\ldots,x_n$. For example, if $Y_1,\ldots,Y_n$ are independent, with $Y_i \sim g(x_i, \theta) + \gam(\beta, 1)$, where $g(x,\theta) = \theta_0 + \sum_{k=1}^p \theta_k x^{k+1}$, then it follows from Propositions~\ref{J.shortcut} and \ref{prop:repar} that 
\[ \J_n(\theta; u) = \frac{1 + \beta r(\beta)}{\beta \Gamma(\beta)} \sum_{i=1}^{n} \Bigl|\sum_{k=0}^{p}x_i^{k}u_{k+1} \Bigr|^{\beta}. \]
The Hellinger information's dependence on the covariates $(x_1,\ldots,x_n)$ is what makes our theory of optimal design possible.  


In what follows, we focus exclusively on the case of $\psi(\theta) = \theta$, and the direction-free definition of Hellinger distance in \eqref{eq:J.without.u.special}, though this is only for simplicity.  The same derivations can be carried out with any specific interest parameter $\psi(\theta)$ in mind.  

Following the now-standard approximate design theory put forth by \citet{kiefer1974general}, let $\xi$ denote a discrete probability measure defined on the design space---the space where the covariates $x_i$ live---with at most $m$ distinct atoms, representing the design itself.  That is, the atoms of $\xi$ represent the specific design points, and the probabilities correspond to the weights (more details below).  Next, with a slight abuse of our previous notation, we write $\J_{\xi}(\theta; u)$ to indicate that the Hellinger information in the direction $u$ depends on the design $\xi$ through the specific covariate values. For example, given design $\xi=\{(w_i, x_i): i=1,...,m\}$, then $\J_{\xi}(\theta; u)=\sum_{i=1}^{m} w_i J_{i}(\theta;u)$, where $J_{i}(\theta;u)$ is the Hellinger information in the direction $u$ based on one observation taken at location $x_i$. Following \eqref{eq:J.without.u.special}, the Hellinger information based on design $\xi$ is defined as 
\[ \J_{\xi}(\theta) = \inf_u \J_{\xi}(\theta; u). \]
Naturally, the optimal design under this setup would be defined as the one that maximizes this measure of information.  

\begin{defn}
\label{def:optimal}
Under the non-regular model setup presented above, the optimal design $\xi^\star$ is one which maximizes the Hellinger information, i.e., 
\[ \xi^\star = \arg\max_\xi \J_{\xi}(\theta). \]
\end{defn}

For comparison to the classical design theory, property \eqref{eq:special.bound} implies that our optimal design in Definition~\ref{def:optimal}, under a regular model, corresponds to an E-optimal design, one that maximizes the minimum eigenvalue of the Fisher information matrix.  For the non-regular case, however, we do not have an information matrix, so it is not clear if other common notions of optimality, such as A- and D-optimality, have any meaning.  For example, non-regularity will cause sampling distributions of estimators to be non-ellipsoidal, so we cannot expect the determinant of some information matrix to correspond to the volume of a confidence ellipsoid. 


Definition~\ref{def:optimal} formulates a new class of optimal design problems, deserving further attention.  As discussed briefly in Section~\ref{S:intro}, there is now a substantial literature on theory and computation related to the optimal design problem in regular cases, and we hope that this paper stimulates a parallel line of work with similar developments for non-regular cases.  There are some similarities to the regular case, in particular, the Hellinger information is non-negative and additive like Fisher information.  Also, the map $\xi \mapsto \J_\xi(\theta)$ is concave for fixed $\theta$, i.e., for any two designs $\xi$ and $\xi'$ and any $w \in [0,1]$, 
\begin{equation}
\label{eq:concave}
\J_{w \xi+(1-w)\xi'}(\theta) \geq w \J_{\xi}(\theta)+(1-w)\J_{\xi'}(\theta),
\end{equation} 
which is important for numerical and/or analytical solution of the optimal design problem.  The following gives some first results along these lines.

\subsection{A general result for non-regular polynomial models}
\label{SS:poly}

Motivated by the setup in \citet{smith1994nonregular}, we consider a non-regular model of the form 
\begin{equation}
\label{eq:poly}
y_i = g(x_i, \theta) + \eps_i, \quad i=1,\ldots,n,
\end{equation} 
where $x_i$ are scalars, $g(x,\theta) = \theta_0 + \sum_{k=1}^p \theta_k x^k$ is a degree-$p$ polynomial, $\theta \in \RR^d$, with $d=p+1$, is an unknown parameter, and $\eps_i$ are independent and identically distributed with density $p_0$ given in \eqref{eq:p0} and known shape parameter $\alpha \in [1,2)$.  As is customary \citep[e.g.,][]{koenker2001quantile}, we will insist that the design points be centered at the origin, which puts a constraint on the design itself.  In particular, we will consider the space of designs $\xi$ given by 
\[ \Xi = \{\xi = (w_i, x_i): \textstyle\sum_i w_i x_i = 0, x_i \in [-A,A]\}, \]
i.e., designs on $[-A,A]$ that are ``balanced'' in the sense that the mean $x$ value is 0, where $A > 0$ is fixed and known.  

The following result shows that, among balanced designs, the subclass of symmetric designs is complete in the sense that the maximum information over symmetric designs is the same as that over the larger class of balanced designs.  This implies that the search for an optimal design can be simplified by restricting it to the smaller class of symmetric designs.  

\begin{thm}
\label{thm:complete.class}
Let $\Xi_{\mathrm{sym}} \subset \Xi$ denote the set of all balanced designs that are also symmetric in the sense that if $x$ is a design point, then it assigns equal weight to both $x$ and $-x$.  Then 
\[ \max_{\xi \in \Xi_{\mathrm{sym}}} \J_{\xi}(\theta) = \max_{\xi \in \Xi} \J_{\xi}(\theta). \]
\end{thm}

\begin{proof}
See Appendix~\ref{SS:thm2}. 
\end{proof}

The next section applies this general result to identify optimal designs in some special cases of the non-regular polynomial regression model above.  The two results, Propositions~\ref{prop:linear} and \ref{prop:quadratic}, suggest that there is a de la Garza phenomenon \citep[e.g.,][]{delagarza1954} in the non-regular case as well, which would be an interesting theoretical topic to pursue in future work.

\section{Optimal designs for some non-regular regression models}
\label{S:special}

In this section, we apply the general result in Theorem~\ref{thm:complete.class} to identify optimal designs in two important special cases of the polynomial model, namely, linear and quadratic.  Throughout we assume the model stated in \eqref{eq:poly}, namely, that the regression model has non-negative errors with distribution having density of the form \eqref{eq:p0}, with known shape parameter $\alpha \in [1,2)$.  

\subsection{Linear model}

Consider the linear version of \eqref{eq:poly}, where $g(x,\theta) = \theta_0 + \theta_1 x$.  For linear models we have a strong intuition from the regular case as to what the optimal design might be. It turns out that the same result holds in the non-regular case as well. 
\begin{prop}
\label{prop:linear}
The optimal design $\xi^\star$, according to Definition~\ref{def:optimal}, for the non-regular linear regression model is the symmetric two-point design with weight $\frac12$ on $x=\pm A$.  
\end{prop}

\begin{proof}
See Section~S2.4 in the Supplementary Material. 
\end{proof}

\subsection{Quadratic model}

Consider a quadratic case where $g(x,\theta) = \theta_0 + \theta_1 x + \theta_2 x^2$.  Here we restrict our attention to the case where the errors $\eps_i$ in the model are exponential, $\alpha=1$.  

\begin{prop}
\label{prop:quadratic}
For the quadratic model, with $\alpha=1$ and the balanced design constraint, the optimal design $\xi^\star$, according to Definition~\ref{def:optimal}, is one with three distinct points $\{-A, 0, A\}$ with respective weights $\{\frac{1-\pi}{2}, \pi, \frac{1-\pi}{2}\}$ for some $\pi \in (0,1)$. 
\end{prop}

\begin{proof}
See Section~S2.5 in the Supplementary Material. 
\end{proof}

Although the proof of Proposition~\ref{prop:quadratic} holds only for the $\alpha=1$ case, we expect that the result also holds for $\alpha \in [1,2)$, and the numerical results in Figure~\ref{fig:Quad} (b) support this conjecture.  The practical importance is that it simplifies the search over $\Xi_{\mathrm{sym}}$ to a search over the scalar $\pi \in [0,1]$. The weight at point $\{0\}$ of the optimal design---or the {\em likely} optimal design for the case of $\alpha\in(1,2]$---depends on the value of $A$ and $\alpha$. Based on Proposition~\ref{prop:quadratic} and the definition of Hellinger information, the optimal weight can be obtained by solving the optimization problem
\begin{equation}
\pi_{A}(\alpha)=\arg \max_{\pi\in[0,1]} f(\pi), 
\end{equation}
where $f(\pi) = f_{\alpha,A}(\pi)$ is given by 
\[ f(\pi) = \min_{\|u\|=1} \Bigl\{ \pi|u_1|^{\alpha}+\tfrac{1-\pi}{2}\bigl( |u_1+A u_2+A^2 u_3|^{\alpha}+ |u_1-A u_2+A^2 u_3|^{\alpha} \bigr) \Bigr\}. \]
This search for the optimal weight, $\pi_{A}(\alpha)$, along with that over $u$ on the surface of the unit sphere, can be handled numerically. 

Figure~\ref{Opt.Pi} shows $\alpha \mapsto \pi_A(\alpha)$ for several values of $A$.  In particular, we see that the (likely) optimal designs put more weight on 0 as either $\alpha$ or $A$ increases. Our optimal designs for non-regular regression models have a similar format to their E-optimal counterparts in the regular case.  That is, a regular E-optimal design for quadratic regression over $[-A,A]$ is given by 
\[ \bigl\{(-A,\tfrac{1-w_A}{2}),(0,w_A),(A,\tfrac{1-w_A}{2})\bigr\}, \]
and, for $A$ in $\{1,1.5,2\}$, the corresponding values of $w_A$ are $\{0.6, 0.75, 0.81\}$.  From Figure~\ref{Opt.Pi}, as anticipated by Corollary~\ref{cor:bound}, we observe that for $\alpha=2$, $\pi_A(2)$ matches the weight $w_A$ of the corresponding regular E-optimal design. This is explained by Corollary 1; when $\alpha=2$, optimal design under Hellinger information is the E-optimal design.

Henceforth, we call the regular E-optimal design counterpart of a non-regular model ``regular-optimal." For the non-regular linear  model, based on Proposition~\ref{prop:linear}, the optimal design coincides with the ``regular-optimal" design. In the numerical results presented below, we compare optimal designs of non-regular quadratic models to their ``regular-optimal'' counterparts.

\begin{figure}
    \centering
   \includegraphics[scale=0.4]{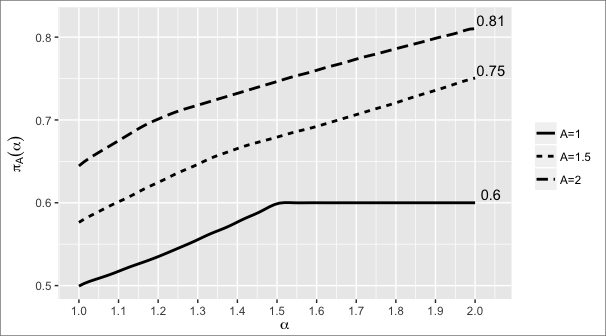}
    \caption{Optimal weight $\pi_{A}(\alpha)$ as a function of $\alpha$ for several $A$ values.}
    \label{Opt.Pi}
\end{figure}

\subsection{Numerical results}
\label{SS:numerical}

Here we show some numerical results to demonstrate the efficiency gain in using the proposed optimal designs over other reasonable designs.  Recall our model is of the form \eqref{eq:poly} with non-negative errors having density \eqref{eq:p0}, with known shape parameter $\alpha \in [1,2)$. 

One complication is that currently there are no results that identify an estimator whose risk attains the lower bound in Theorem~\ref{thm:bound}.  Consequently, we are currently unable to guarantee that minimizing this lower bound will result in improved estimation for any {\em given} estimator.  But we do have a reasonable estimator, described next, and the results below do indicate that the design that minimizes the lower bound in Theorem~\ref{thm:bound} does indeed result in improved efficiency for this particular estimation.  

For the class of non-regular polynomial regression problems in consideration here, \citet{smith1994nonregular} proposed an estimator based on solving a linear programming problem: choosing $(\theta_0,\ldots,\theta_p)$ such that $\theta_0$ is maximized subject to the condition that $y_i \geq \sum_{k=1}^p \theta_k x_i^k$ for each $i=1,\ldots,n$.  This estimator agrees with the maximum likelihood estimator in the case $\alpha=1$, has a $O(n^{-1/\alpha})$ convergence rate, which matches the one given by the lower bound in \eqref{eq:minimax}, and can be readily computed using the {\tt quantreg} package in R \citep{Rquantreg}.  Moreover, as \citet[][p.~174]{smith1994nonregular} argues, it is generally superior to maximum likelihood in non-regular cases. For these reasons, comparisons of designs based on this estimator ought to be informative. 

Figure~\ref{fig:Lin} presents simulation results on the quality of estimation for the Hellinger optimal design versus 5-, 10-, and 15-point uniform designs for the non-regular linear models, while Figure~\ref{fig:Quad} presents simulation results comparing Hellinger optimal design versus 5-point uniform design and the regular-optimal design. The study proceeds as follows.  For each design space $[-A,A]$ and candidate design, the $n$-vector $y$ is simulated from the corresponding model, with the specified value of $\alpha$ and $\theta$, and then Smith's estimator $\hat\theta$ is computed.  Repeat this process 1000 times and compute the Monte Carlo estimate of the risk $R(\hat\theta, \theta)$ as usual.  This risk is the sum of mean square errors for each component of the parameter vector.


\begin{figure}
    \centering
    \subfigure[Linear model, $\alpha=1$]{{\includegraphics[width=12cm]{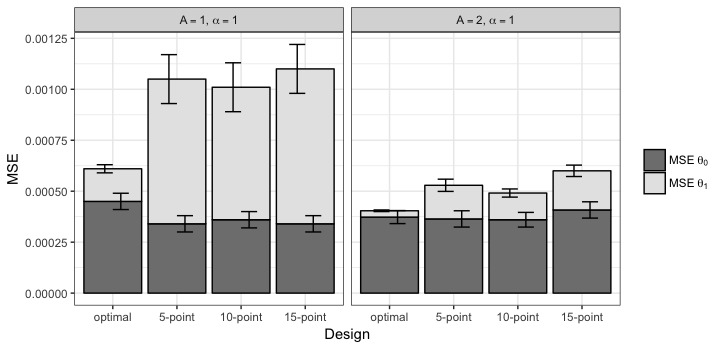}  }}\\
    \subfigure[Linear model, $\alpha=1.4$]{{\includegraphics[width=12cm]{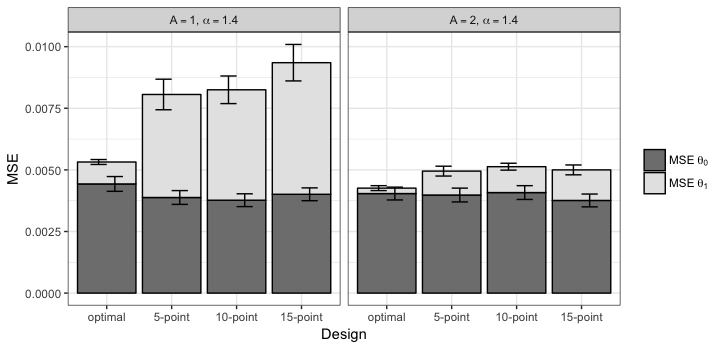} }}%
    \caption{Comparison for non-regular linear model, based on $n=120$ and $\theta=(6, 0.5)$.}%
    \label{fig:Lin}%
\end{figure}

Figure~\ref{fig:Lin} shows that, under different regularity conditions, the optimal design from Proposition~\ref{prop:linear} is superior in terms of risk. In particular, it is significantly better in the estimation of the slope, $\theta_1$, whereas no design performs significantly better than the others in the estimation of the intercept. The results presented in Figure~\ref{fig:Quad}(a) are consistent with Proposition~\ref{prop:quadratic} in the case of $\alpha=1$.  In each case, the optimal design performs significantly better than both the 5-point uniform design and the regular-optimal design, despite the similarity of the optimal and regular-optimal designs in terms of weight at point 0. Similarly, Figure~\ref{fig:Quad}(b) supports our intuition that Proposition~\ref{prop:quadratic} can be extended to cases with $\alpha>1$.

\begin{figure}
    \centering
    \subfigure[Left panel: $\pi_1(1)=0.5$ and the regular-optimal design is $\{(-1,0.2),(0,0.6),(1,0.2)\} $; Right panel: $\pi_2(1)=0.75$ and the regular-optimal design is $\{(-2,0.095),(0,0.81),(2,0.095)\}$]
    {{\includegraphics[width=12cm]{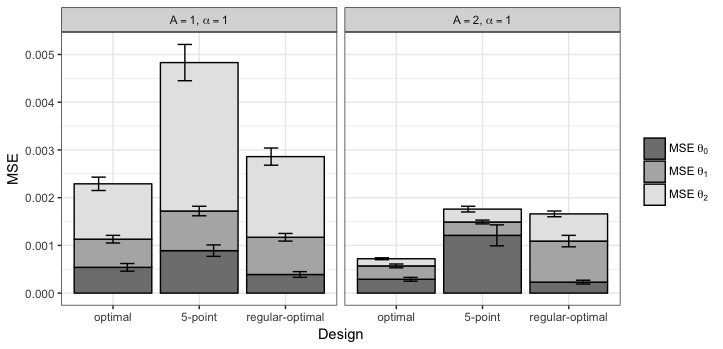}  }}\\
    \subfigure[Left panel: $\pi_{1.5}(1.1)=0.6$ and the  regular-optimal design is $\{(-1.5,0.125),(0,0.75),(1.5,0.125)\} $; Right panel: $\pi_2(1.5)=0.75$ and the regular-optimal design is : $\{(-2,0.095),(0,0.81),(2,0.095)\}$]
    {{\includegraphics[width=12cm]{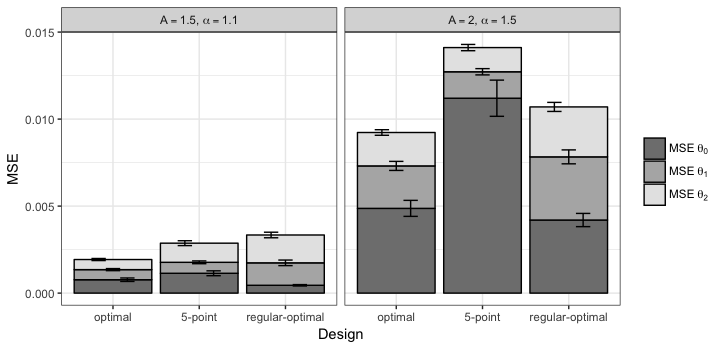} }}%
    \caption{Comparison for non-regular quadratic model, based on $n=120$ and $\theta=(2, 4, 0.8)$.}%
    \label{fig:Quad}%
\end{figure}

\section{Conclusion}
\label{S:discuss}

This paper aims to establish a framework for optimal design in the context of non-regular models where the Fisher information matrix does not exist.  Towards this goal, we defined an alternative measure of information, based on a local approximation of the squared Hellinger distance between models, suitable for non-regular problems.  The proposed Hellinger information has some close connection to the Fisher information when both exist and, more generally, the former has many of the familiar properties of the latter.  In particular, in Theorem~\ref{thm:bound} we establish a parallel to the classical Cram\'er--Rao inequality which connects our proposed Hellinger information measure to the quality of estimators.  This naturally leads to a notion of optimal designs in non-regular problems, i.e., the ``optimal design'' is one that minimizes the lower bound in Theorem~\ref{thm:bound}.  

The proposed optimal design framework introduces a new class of optimization problems to solve, what we have considered here is only the tip of the iceberg.  However, the tools currently available in the optimal design literature for regular problems are expected to be useful here.  For example, in a particular non-regular polynomial regression setting, we establish a theorem to simplify the numerical and/or analytical search for a particular optimal design, and we apply this general result in the linear and quadratic cases.  Developing the theory and computational methods to handle more complex non-regular models, as well as identifying estimators that attain the lower bound \eqref{eq:minimax}, are interesting topics for future investigation.  

Aside from creating a new class of design problems to be investigated, the developments here also shed light on how much our current understanding of design problems depends on the regularity of the models being considered.  That is, beyond its value in helping us tackle specific cases in which regularity conditions do not apply, the study of non-regular problems also deepens our understanding of regularity itself and how it affects optimal design. For example, questions about the type of optimality criterion to consider (e.g., A- versus D- versus E-optimal) are apparently only relevant for those regular cases where the Fisher information matrix is exactly or approximately related to the dispersion matrix of an estimator. While this paper provides some important insights about non-regular models and corresponding optimal design problems, there is still much more to be done.

\section*{Acknowledgments}

The authors are grateful for the helpful suggestions from the Associate Editor referees, and Professor Arkady Shemyakin.  The authors also thank Mr.~Zhiqiang Ye for pointing out a mistake in the statement and proof of Proposition~\ref{J.shortcut} in a previous version.

\appendix

\section{Proofs of theorems}
\label{S:proofs1}

\subsection{Proof of Theorem~\ref{thm:bound}}
\label{SS:thm1}

The proof requires a connection between Hellinger distance and risk of an estimator. This first step is based in part on Section~I.6 of \citet{ibragimov1981statistical}, although our setup and conclusions are more general in certain ways. We summarize this in the following lemma, proved in the Supplementary Material.  

\begin{lem}
\label{lem:ih}
For data $Y \in \YY$, consider a model $P_\theta$, with $\mu$-density $p_\theta$, indexed by a parameter $\theta \in \Theta \subseteq \RR^d$.  Let $\psi = \psi(\theta)$ be the interest parameter, where $\psi: \RR^d \to \RR^q$.  For an estimator $T=T(Y)$ of $\psi$, the risk function $R_\psi(T,\theta)$ for the estimator $T$ satisfies 
\[ R_\psi(T,\theta) + R_\psi(T,\vartheta) \geq \min\Bigl\{ \frac{1-h(\theta;\vartheta)}{4h(\theta; \vartheta)}, \frac{1}{16} \Bigr\} \|\psi(\theta) - \psi(\vartheta)\|^2. \]
\end{lem}

For the proof of Theorem~\ref{thm:bound}, start with the squared Hellinger distance between joint distributions $P_\theta^n$ and $P_\vartheta^n$, given by 
\[ h^n(\theta; \vartheta) := H^2(P_\theta^n, P_\vartheta^n) = 2 \Bigl[ 1 - \prod_{i=1}^n \Bigl\{ 1 - \frac{h_i(\theta; \vartheta)}{2} \Bigr\} \Bigr], \]
where $h_i(\theta; \vartheta) = H^2(P_{i,\theta}, P_{i,\vartheta})$ is the squared Hellinger distance between individual components.  If $\theta$ and $\vartheta$ are sufficiently close, in the sense that $h_i(\theta; \vartheta) \leq 1$ for each $i=1,\ldots,n$, then, given the following inequalities, 
\[ 1-x \leq -\log x \quad \text{and} \quad -\log(1-x) \leq 2x, \quad x \in [0,1/2], \]
it follows that 
\begin{equation}\label{additive}
h^n(\theta; \vartheta) \leq -2\sum_{i=1}^n \log\Bigl\{1 - \frac{h_i(\theta; \vartheta)}{2} \Bigr\} \leq 2 \sum_{i=1}^n h_i(\theta; \vartheta).
\end{equation}
According to our assumption about local expansion of the individual $h_i$'s, if $\vartheta = \theta + \eps \, u$ for a unit vector $u$, then 
\[ h^n(\theta; \theta + \eps \, u) \leq 2 \J_n(\theta; u) \, \eps^\alpha + o(n\eps^\alpha), \quad \eps \to 0. \]
When we take $\eps$ equal to $\eps_{n,u} = \{3 \J_n(\theta; u)\}^{-1/\alpha}$, then we get 
\[ h^n(\theta; \theta + \eps_{n,u} \, u) \leq \tfrac23 + o(1), \quad n \to \infty, \]
where the latter ``$o(1)$'' conclusion is justified by the assumption \eqref{eq:accumulates} about the rate of information accumulation.  Therefore, for large enough $n$, with $\vartheta_{n,u} = \theta + \eps_{n,u} \, u$, $h^n(\theta; \vartheta_{n,u}) \leq \frac34$, it follows from the above lemma that 
\[ R_\psi(T_n, \theta) + R_\psi(T_n, \vartheta_{n,u}) \geq \tfrac{1}{16} \|\psi(\theta) - \psi(\vartheta_{n,u})\|^2. \]
Since $\psi$ is differentiable, there is a Taylor approximation at $\theta$:
\[ \psi(\theta) - \psi(\vartheta_{n,u}) = D_\psi(\theta) (\theta-\vartheta_{n,u}) + o(\|\theta-\vartheta_{n,u}\|), \]
where the latter little-oh means a $q$-vector whose entries are all of that magnitude.  Plugging in the definition of $\vartheta_{n,u}$ gives 
\[ \psi(\theta) - \psi(\theta + \eps_{n,u} \, u) = -\eps_{n,u} D_\psi(\theta) \, u + o(\eps_{n,u}), \quad n \to \infty, \]
and, hence, 
\[ \|\psi(\theta) - \psi(\theta + \eps_{n,u} \, u)\|^2 = \eps_{n,u}^2 \|D_\psi(\theta) \, u + o(1)\|^2 \geq \tfrac12 \eps_{n,u}^2 \|D_\psi(\theta) \, u\|^2. \]
Plugging in the definition of $\eps_{n,u}$ establishes that
\[ R_\psi(T_n, \theta + \eps_{n,u} u) + R_\psi(T_n, \theta) \gtrsim \|D_\psi(\theta) \, u\|^2 \, \J_n(\theta; u)^{-2/\alpha}. \]
Also, the constant that has been absorbed in ``$\gtrsim$'' is $(32)^{-1} 3^{-2/\alpha}$.  Finally, the claim \eqref{eq:minimax} follows from the above display and the general fact that, for a function $f$ defined on a set $A$, $f(y_1) + f(y_2)$ is smaller than $2 \sup_A f(y)$.

\subsection{Proof of Theorem~\ref{thm:complete.class}}
\label{SS:thm2}

Take any fixed design $\xi = \{(w_m, x_m): m=1,\ldots,M\}$, and define a function
\[ L(u; x) = \J_{\xi}(\theta; u) = \sum_{m=1}^M w_m \Bigl| \sum_{k=0}^p x_m^k u_{k+1} \Bigr|^\alpha. \]
The $L$ function does not depend on $\theta$ because it is based on the information in a location parameter problem, but it does depend implicitly on the $w$ component of the design $\xi$.  From the trivial identity, 
\[ a \, x_m^k = a \, (-1)^k \, (-x_m)^k, \quad \text{any $a \in \RR$, any $m$, and any $k$}, \]
it follows immediately that $L(u; x) = L(v; -x)$, for any unit vector $u \in \RR^{p+1}$, where $v_{k+1} = (-1)^k u_{k+1}$, $k=0,\ldots,p$.  Since this new vector $v$ is also a unit vector, we have 
\[ \min_u L(u; x) = \min_v L(v; -x). \]
This implies that the reflected design $\xi'$---the one that replaces the original $x_m$ in $\xi$ with $-x_m$, but keeps the same weights---satisfies $\J_{\xi}(\theta) = \J_{\xi'}(\theta)$.  Define the mixture design $\xi^\dagger = \frac12 \xi + \frac12 \xi'$, which is symmetric by construction, and by concavity \eqref{eq:concave} satisfies
\begin{align*}
\J_{\xi^\dagger}(\theta) & = \min_u \bigl\{ \tfrac12 \J_{\xi}(\theta; u) + \tfrac12 \J_{\xi'}(\theta; u) \bigr\} \\
& \geq \tfrac12 \min_u \J_{\xi}(\theta; u) + \tfrac12 \min_u \J_{\xi'}(\theta; u).
\end{align*}
We showed above that the two terms in the lower bound are equal and, consequently, $\J_{\xi^\dagger}(\theta) \geq \J_{\xi}(\theta)$.  Therefore, for any design $\xi$ there exists a symmetric design with Hellinger information at least as big; hence, symmetric designs form a complete class.


\ifthenelse{1=1}{
\bibliographystyle{apalike}
\bibliography{uictest}

\begin{thebibliography}{}

\bibitem[Bernardo, 1979]{bernardo1979}
Bernardo, J.-M. (1979).
\newblock Reference posterior distributions for {B}ayesian inference.
\newblock {\em J. Roy. Statist. Soc. Ser. B}, 41(2):113--147.
\newblock With discussion.

\bibitem[Biedermann et~al., 2009]{biedermann2009constrained}
Biedermann, S., Dette, H., and Hoffmann, P. (2009).
\newblock Constrained optimal discrimination designs for {F}ourier regression
  models.
\newblock {\em Ann. Inst. Statist. Math.}, 61(1):143--157.

\bibitem[Biedermann et~al., 2011]{biedermann2011optimal}
Biedermann, S., Dette, H., and Woods, D.~C. (2011).
\newblock Optimal design for additive partially nonlinear models.
\newblock {\em Biometrika}, 98(2):449--458.

\bibitem[Biedermann et~al., 2006]{biedermann2006optimal}
Biedermann, S., Dette, H., and Zhu, W. (2006).
\newblock Optimal designs for dose-response models with restricted design
  spaces.
\newblock {\em J. Amer. Statist. Assoc.}, 101(474):747--759.

\bibitem[Chernozhukov and Hong, 2004]{chernozhukov2004likelihood}
Chernozhukov, V. and Hong, H. (2004).
\newblock Likelihood estimation and inference in a class of nonregular
  econometric models.
\newblock {\em Econometrica}, 72(5):1445--1480.

\bibitem[Cousineau, 2009]{cousineau2009fitting}
Cousineau, D. (2009).
\newblock Fitting the three-parameter {W}eibull distribution: Review and
  evaluation of existing and new methods.
\newblock {\em IEEE Transactions on Dielectrics and Electrical Insulation},
  16(1):281--288.

\bibitem[de~la Garza, 1954]{delagarza1954}
de~la Garza, A. (1954).
\newblock Spacing of information in polynomial regression.
\newblock {\em Ann. Math. Statistics}, 25:123--130.

\bibitem[Dette et~al., 2008]{dette2008optimal}
Dette, H., Bretz, F., Pepelyshev, A., and Pinheiro, J. (2008).
\newblock Optimal designs for dose-finding studies.
\newblock {\em J. Amer. Statist. Assoc.}, 103(483):1225--1237.

\bibitem[Dette et~al., 2018]{dette2018optimal}
Dette, H., Guchenko, R., Melas, V.~B., and Wong, W.~K. (2018).
\newblock Optimal discrimination designs for semiparametric models.
\newblock {\em Biometrika}, 105(1):185--197.

\bibitem[Dette et~al., 2017]{dette2017new}
Dette, H., Konstantinou, M., and Zhigljavsky, A. (2017).
\newblock A new approach to optimal designs for correlated observations.
\newblock {\em Ann. Statist.}, 45(4):1579--1608.

\bibitem[Dette and Melas, 2011]{dette2011note}
Dette, H. and Melas, V.~B. (2011).
\newblock A note on the de la {G}arza phenomenon for locally optimal designs.
\newblock {\em Ann. Statist.}, 39(2):1266--1281.

\bibitem[Dette et~al., 2016]{dette2016optimal}
Dette, H., Pepelyshev, A., and Zhigljavsky, A. (2016).
\newblock Optimal designs in regression with correlated errors.
\newblock {\em Ann. Statist.}, 44(1):113--152.

\bibitem[Dette and Schorning, 2013]{dette2013complete}
Dette, H. and Schorning, K. (2013).
\newblock Complete classes of designs for nonlinear regression models and
  principal representations of moment spaces.
\newblock {\em Ann. Statist.}, 41(3):1260--1267.

\bibitem[Dette and Titoff, 2009]{dette2009optimal}
Dette, H. and Titoff, S. (2009).
\newblock Optimal discrimination designs.
\newblock {\em Ann. Statist.}, 37(4):2056--2082.

\bibitem[Dror and Steinberg, 2006]{dror2006robust}
Dror, H.~A. and Steinberg, D.~M. (2006).
\newblock Robust experimental design for multivariate generalized linear
  models.
\newblock {\em Technometrics}, 48(4):520--529.

\bibitem[Feller et~al., 2017]{feller2017optimal}
Feller, C., Schorning, K., Dette, H., Bermann, G., and Bornkamp, B. (2017).
\newblock Optimal designs for dose response curves with common parameters.
\newblock {\em Ann. Statist.}, 45(5):2102--2132.

\bibitem[Gotwalt et~al., 2009]{gotwalt2009fast}
Gotwalt, C.~M., Jones, B.~A., and Steinberg, D.~M. (2009).
\newblock Fast computation of designs robust to parameter uncertainty for
  nonlinear settings.
\newblock {\em Technometrics}, 51(1):88--95.

\bibitem[Harman and Benkov\'a, 2017]{harman2017barycentric}
Harman, R. and Benkov\'a, E. (2017).
\newblock Barycentric algorithm for computing {$D$}-optimal size- and
  cost-constrained designs of experiments.
\newblock {\em Metrika}, 80(2):201--225.

\bibitem[Hirose and Lai, 1997]{hirose1997inference}
Hirose, H. and Lai, T.~L. (1997).
\newblock Inference from grouped data in three-parameter {W}eibull models with
  applications to breakdown-voltage experiments.
\newblock {\em Technometrics}, 39(2):199--210.

\bibitem[Ibragimov and Hasminskii, 1981]{ibragimov1981statistical}
Ibragimov, I.~A. and Hasminskii, R.~Z. (1981).
\newblock {\em Statistical Estimation}, volume~16 of {\em Applications of
  Mathematics}.
\newblock Springer-Verlag, New York-Berlin.
\newblock Asymptotic theory, Translated from the Russian by Samuel Kotz.

\bibitem[Kiefer, 1974]{kiefer1974general}
Kiefer, J. (1974).
\newblock General equivalence theory for optimum designs (approximate theory).
\newblock {\em Ann. Statist.}, 2:849--879.

\bibitem[Koenker, 2013]{Rquantreg}
Koenker, R. (2013).
\newblock {\em quantreg: Quantile Regression}.
\newblock R package version 5.05.

\bibitem[Koenker and Hallock, 2001]{koenker2001quantile}
Koenker, R. and Hallock, K. (2001).
\newblock Quantile regression: An introduction.
\newblock {\em Journal of Economic Perspectives}, 15(4):43--56.

\bibitem[Lehmann and Casella, 1998]{lehmann.casella.book}
Lehmann, E.~L. and Casella, G. (1998).
\newblock {\em Theory of {P}oint {E}stimation}.
\newblock Springer Texts in Statistics. Springer-Verlag, New York, second
  edition.

\bibitem[Lin et~al., 2018]{nonreg.supp}
Lin, Y., Martin, R., and Yang, M. (2018).
\newblock Supplement to ``on optimal designs for non-regular models''.
\newblock DOI...

\bibitem[Lindley, 1956]{lindley1956}
Lindley, D.~V. (1956).
\newblock On a measure of the information provided by an experiment.
\newblock {\em Ann. Math. Statist.}, 27:986--1005.

\bibitem[L\'opez-Fidalgo et~al., 2007]{lopez2007optimal}
L\'opez-Fidalgo, J., Tommasi, C., and Trandafir, P.~C. (2007).
\newblock An optimal experimental design criterion for discriminating between
  non-normal models.
\newblock {\em J. R. Stat. Soc. Ser. B Stat. Methodol.}, 69(2):231--242.

\bibitem[Pollard, 1997]{pollard1997another}
Pollard, D. (1997).
\newblock Another look at differentiability in quadratic mean.
\newblock In {\em Festschrift for {L}ucien {L}e {C}am}, pages 305--314.
  Springer, New York.

\bibitem[Pollard, 2005]{pollard.chap6}
Pollard, D. (2005).
\newblock {\em Asymptopia}.
\newblock Chapter 6 on ``Hellinger differentiability,''
  \url{http://www.stat.yale.edu/~pollard/Courses/607.spring05/handouts/DQM.pdf}.

\bibitem[Sagnol and Harman, 2015]{sagnol2015computing}
Sagnol, G. and Harman, R. (2015).
\newblock Computing exact {$D$}-optimal designs by mixed integer second-order
  cone programming.
\newblock {\em Ann. Statist.}, 43(5):2198--2224.

\bibitem[Schorning et~al., 2017]{schorning2017optimal}
Schorning, K., Dette, H., Kettelhake, K., Wong, W.~K., and Bretz, F. (2017).
\newblock Optimal designs for active controlled dose-finding trials with
  efficacy-toxicity outcomes.
\newblock {\em Biometrika}, 104(4):1003--1010.

\bibitem[Shemyakin, 2014]{shemyakin2014hellinger}
Shemyakin, A. (2014).
\newblock Hellinger distance and non-informative priors.
\newblock {\em Bayesian Anal.}, 9(4):923--938.

\bibitem[Shemyakin, 1991]{shemyakin1991}
Shemyakin, A.~E. (1991).
\newblock Multidimensional integral inequalities of {R}ao-{C}ram\'er type for
  parametric families with singularities.
\newblock {\em Sibirsk. Mat. Zh.}, 32(4):204--215, 230.

\bibitem[Shemyakin, 1992]{shemyakin1992}
Shemyakin, A.~E. (1992).
\newblock On information inequalities in parametric estimation theory.
\newblock {\em Teor. Veroyatnost. i Primenen.}, 37(1):121--123.

\bibitem[Smith, 1985]{smith1985}
Smith, R.~L. (1985).
\newblock Maximum likelihood estimation in a class of nonregular cases.
\newblock {\em Biometrika}, 72(1):67--90.

\bibitem[Smith, 1994]{smith1994nonregular}
Smith, R.~L. (1994).
\newblock Nonregular regression.
\newblock {\em Biometrika}, 81(1):173--183.

\bibitem[van~der Vaart, 2002]{vaart.lecam.2002}
van~der Vaart, A. (2002).
\newblock The statistical work of {L}ucien {L}e {C}am.
\newblock {\em Ann. Statist.}, 30(3):631--682.
\newblock Dedicated to the memory of Lucien Le Cam.

\bibitem[van~der Vaart, 1998]{van1998asymptotic}
van~der Vaart, A.~W. (1998).
\newblock {\em Asymptotic Statistics}.
\newblock Cambridge University Press, Cambridge.

\bibitem[Waterhouse et~al., 2008]{waterhouse2008design}
Waterhouse, T.~H., Woods, D.~C., Eccleston, J.~A., and Lewis, S.~M. (2008).
\newblock Design selection criteria for discrimination/estimation for nested
  models and a binomial response.
\newblock {\em J. Statist. Plann. Inference}, 138(1):132--144.

\bibitem[Woodroofe, 1974]{woodroofe1974}
Woodroofe, M. (1974).
\newblock Maximum likelihood estimation of translation parameter of truncated
  distribution. {II}.
\newblock {\em Ann. Statist.}, 2:474--488.

\bibitem[Yang, 2010]{yang2010garza}
Yang, M. (2010).
\newblock On the de la {G}arza phenomenon.
\newblock {\em Ann. Statist.}, 38(4):2499--2524.

\bibitem[Yang et~al., 2013]{yang.biedermann.tang.2013}
Yang, M., Biedermann, S., and Tang, E. (2013).
\newblock On optimal designs for nonlinear models: a general and efficient
  algorithm.
\newblock {\em J. Amer. Statist. Assoc.}, 108(504):1411--1420.

\bibitem[Yang and Stufken, 2009]{yang2009support}
Yang, M. and Stufken, J. (2009).
\newblock Support points of locally optimal designs for nonlinear models with
  two parameters.
\newblock {\em Ann. Statist.}, 37(1):518--541.

\bibitem[Yang and Stufken, 2012]{yang2012identifying}
Yang, M. and Stufken, J. (2012).
\newblock Identifying locally optimal designs for nonlinear models: a simple
  extension with profound consequences.
\newblock {\em Ann. Statist.}, 40(3):1665--1681.

\bibitem[Yu, 2011]{yu2011d}
Yu, Y. (2011).
\newblock D-optimal designs via a cocktail algorithm.
\newblock {\em Stat. Comput.}, 21(4):475--481.

\end{thebibliography}
}{}

\pagebreak

\section*{S \quad Supplementary material}

\subsection*{S1. A multi-parameter example}

As an illustrative example, consider the case where $\theta=(\theta_1,\theta_2)$ is two-dimensional and $P_\theta$ is $\unif(\theta_1, \theta_1 + \theta_2)$, where $\theta_1 \in \RR$ and $\theta_2 > 0$.  It is not difficult to show that 
\[ h(\theta, \vartheta) = 2\Bigl\{ 1 - \frac{(\theta_1 + \theta_2) \wedge (\vartheta_1 + \vartheta_2) - \theta_1 \vee \vartheta_1}{(\theta_2 \vartheta_2)^{1/2}} \Bigr\}. \]
Writing $\vartheta = \theta + \eps u$ for a unit vector $u=(u_1,u_2)$, and by considering all the possible configurations of $u$, it can be shown that $\alpha=1$ and 
\[ J(\theta; u) = \theta_2^{-1} g(u), \]
where $g(u)$ is a function that depends only on $u$, not on $\theta$.  This expression is consistent with what we would expect from the familiar Fisher information, since $\theta_1$ and $\theta_2$ are, in this formulation, location and scale parameters, respectively.  The function $g(u)$ is not complicated, just that the expression varies depending on where on the unit circle $u$ is.  For example, if $u$ is in the first or third quadrants, then 
\[ g(u) = |2u_1 + u_2|. \]
To derive the corresponding expressions for other values of $u$, the second and fourth quadrants need to be split in half along the line $y=-x$.  Figure~\ref{fig:Jplot}(a) shows a plot of $\omega \mapsto J(\theta; u_\omega)$, where $u_\omega = (\cos \omega, \sin \omega)$, as $\omega$ varies over $(0, 2\pi)$, for several $\theta$ values.  Notice that the shape of the function does not depend on $\theta$, only the scale, which means the value of $u$ that minimizes $J(\theta; u)$ does not depend on $\theta$.  This plot also reveals the symmetry with respect to reflections $u \to -u$ through the origin.  

To gain some intuition about the somewhat complicated lower bound established in Theorem~1, suppose that our goal is to estimate the scale parameter $\theta_2$, i.e., $\psi(\theta) = \theta_2$.  Then the relevant Hellinger information is 
\[ J^\psi(\theta; u) = \frac{J(\theta; u)}{|u_2|} = \frac{g(u)}{\theta_2 |u_2|}. \]
Figure~\ref{fig:Jplot}(b) plots $\omega \mapsto J^\psi(\theta; u_\omega)$ for three different $\theta$ values.  The minimum value of these functions would be the relevant Hellinger information for estimating $\theta_2$, and the horizontal lines drawn there correspond to $\theta_2^{-1}$.  Given independent data $Y_1,\ldots,Y_n$, an oracle who knows the value of $\theta_1$ and uses the maximum likelihood estimator of $\theta_2$ when $\theta_1$ is fixed and known, would have mean square error of the order $n^{-2} \theta_2^2$, which agrees with the theorem's lower bound up to constants.  

\begin{figure}
\centering
\subfigure[Plot of $u \mapsto J(\theta;u)$]{\scalebox{0.55}{\includegraphics{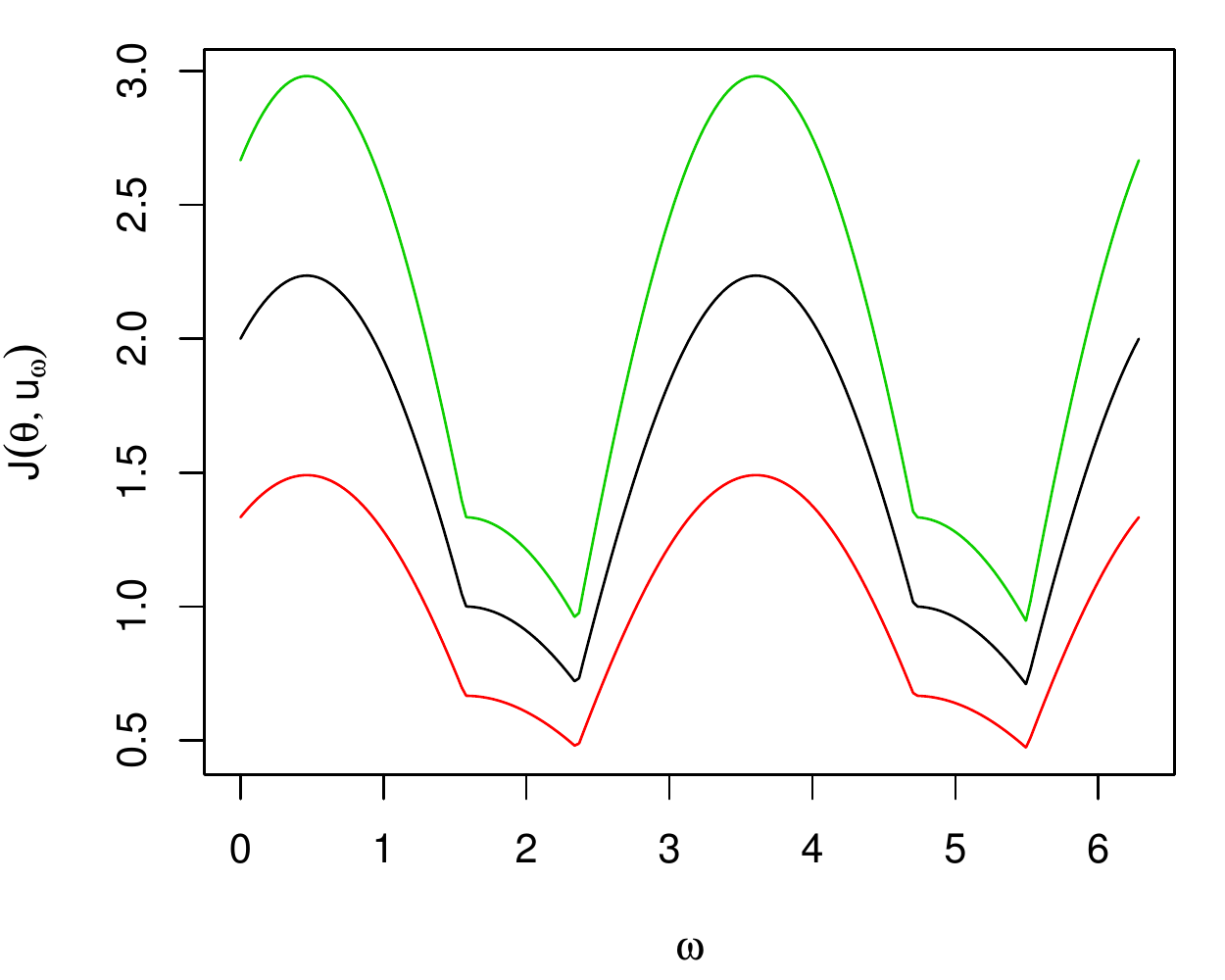}}}
\subfigure[Plot of $u \mapsto J^\psi(\theta; u)$]{\scalebox{0.55}{\includegraphics{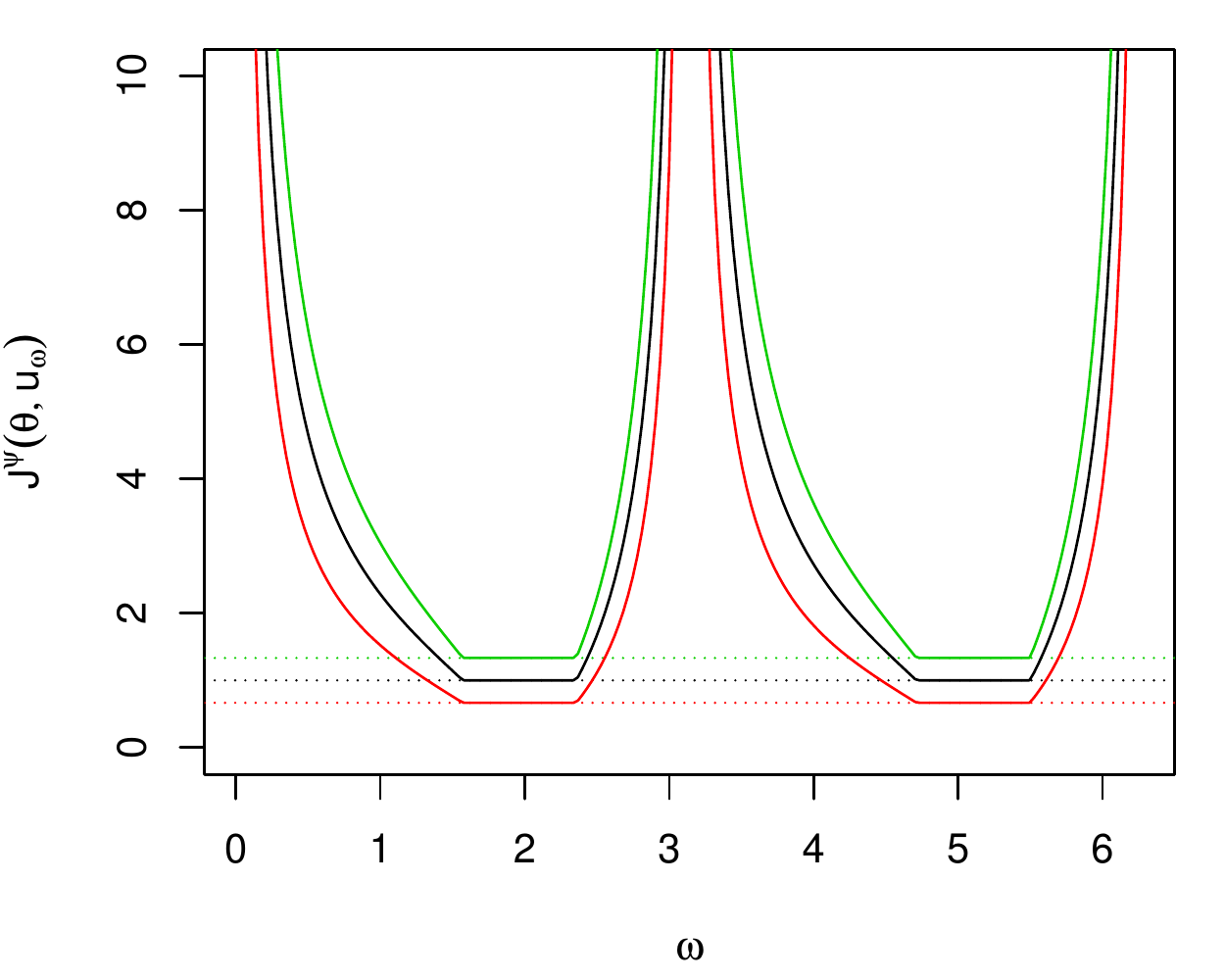}}}
\caption{Plots of the Hellinger information for the two-dimensional uniform example. The black, red, and green lines correspond to $\theta_2=1$, $\theta_2=1.5$ and $\theta_2=0.75$, respectively.}
\label{fig:Jplot}
\end{figure}

\subsection*{S2. Remaining proofs}

\subsubsection*{S2.1. Proof of Proposition~1}

Let $f(y)$ be a function such that $f(0)=c$ and the density $p_0$ in Equation~(1) of the main text satisfies 
 \[ p_0(y) = \beta \, f(y) \, y^{\beta - 1}, \quad y > 0. \]
Without loss of generality, consider $\eps > 0$ and small.  For the squared Hellinger distance, we have 
\begin{align*}
h(\theta, \theta + \eps) & = \int_{-\infty}^\infty \{ p_{\theta+\eps}^{1/2}(y) - p_\theta^{1/2}(y) \}^2 \,dy \\
& = \int_\theta^{\theta + \eps} p_\theta(y) \, dy + \int_{\theta + \eps}^\infty \{ p_\theta^{1/2}(y) - p_{\theta + \eps}^{1/2}(y) \}^2 \, dy\\
& = \prob_\theta(\theta \leq Y \leq \theta + \eps) + \int_{\theta + \eps}^\infty \{ p_\theta^{1/2}(y) - p_{\theta + \eps}^{1/2}(y) \}^2 \, dy.
\end{align*}
The first term equals $\prob_0(Z \leq \eps)$ and it is easy to see that this is $f(0) \eps^\beta + o(\eps^\beta)$.  For the second term, we make a change of variable, $z=y-\theta-\eps$, so that 
\[ \int_{\theta + \eps}^\infty \{ p_\theta^{1/2}(y) - p_{\theta + \eps}^{1/2}(y) \}^2 \, dy = \int_0^\infty \{p_0^{1/2}(z + \eps) - p_0^{1/2}(z)\}^2 \,dz. \]
If we split up this latter integral as 
\[ \int_0^\Delta \{p_0^{1/2}(z + \eps) - p_0^{1/2}(z)\}^2 \,dz + \int_\Delta^\infty \{p_0^{1/2}(z + \eps) - p_0^{1/2}(z)\}^2 \,dz, \]
for $\Delta > 0$ as in the statement of the proposition, then the right-most integral is $O(\eps^2) = o(\eps^\beta)$ by Equation~(4) in the main text and the dominated convergence theorem.  It remains to investigate the left-most integral, which equals 
\[ \beta \int_0^\Delta \{f(z+\eps)^{1/2}(z+\eps)^{(\beta-1)/2} - f(z)^{1/2} z^{(\beta-1)/2}\}^2 \,dz. \]
We proceed by adding and subtracting $f(z+\eps)^{1/2} z^{(\beta-1)/2}$ inside $\{\cdots\}$, so that the new integrand looks like 
\[ \{f(z+\eps)^{1/2}(z+\eps)^{(\beta-1)/2} - f(z)^{1/2} z^{(\beta-1)/2}\}^2 = \sum_{i=1}^3 I_i(z; \eps), \] 
where 
\begin{align*}
I_1(z; \eps) & = f(z+\eps) \{ (z+\eps)^{(\beta-1)/2} - z^{(\beta-1)/2} \}^2 \\
I_2(z; \eps) & = \{ f(z+\eps)^{1/2} - f(z)^{1/2} \}^2 z^{(\beta-1)/2} \\
I_3(z; \eps) & = 2 f(z+\eps)^{1/2} \{ f(z+\eps)^{1/2} - f(z)^{1/2} \} \{ (z+\eps)^{(\beta-1)/2} - z^{(\beta-1)/2} \}. 
\end{align*}
The second term, $I_2$, is the easiest to deal with, so we take this one first.  Because $f$ is smooth and slowly varying near zero, the mean value theorem says that $f(z+\eps)^{1/2} - f(z)^{1/2} \lesssim \eps$, which implies that 
\[ \int_0^\Delta I_2(z; \eps) \,dz \lesssim \eps^2 \int_0^\Delta z^{(\beta-1)/2} \,dz \lesssim \eps^2 = o(\eps^\beta), \quad \eps \to 0. \]
The third term, $I_3$, is similar.  That is, after applying the mean value theorem to both of the differences in $I_3$, we have that 
\[ \int_0^\Delta I_3(z; \eps) \,dz \lesssim \eps^2 \int_0^\Delta z^{-(3-\beta)/2} \,dz \]
and, since the integral converges, the upper bound is $O(\eps^2) = o(\eps^\beta)$ as $\eps \to 0$.  It remains to deal with the $I_1$ term, namely, 
\[ \beta \int_0^\Delta f(z+\eps) \{(z+\eps)^b - z^b\}^2 \, dz, \quad \text{where} \quad b=\tfrac{\beta-1}{2}. \]
Make a change-of-variable, $w=z/\eps$, so that the above integral becomes 
\[ \beta \eps^\beta \int_0^{\Delta/\eps} f(\eps w + \eps) \{(w + 1)^b - w^b\}^2 \, dw. \]
By the mean value theorem, we have that $\{(w+1)^b - w^b\}^2 \leq \min\{1, b^2w^{2(b-1)}\}$, and, since $2(b-1) = \beta-3 < -1$, this upper bound is integrable over $w \in (0,\infty)$.  Since $f$ is also bounded, it follows from dominated convergence theorem that, as $\eps \to 0$, 
\[ \int_0^{\Delta/\eps} f(\eps w + \eps) \{(w + 1)^b - w^b\}^2 \, dw \to f(0) \int_0^\infty \{(w + 1)^b - w^b\}^2 \, dw. \]
The integral on the right-hand side is to be recognized as $r(\beta)$ in Equation~(5) of the main text.  Putting everything together, we have 
\[ h(\theta, \theta + \eps) = f(0)\{1 + \beta \, r(\beta)\} \eps^\beta + o(\eps^\beta), \]
which implies that the regularity index is $\alpha=\beta$ and the Hellinger information is 
\[ J(\theta) := \lim_{\eps \to 0} \frac{h(\theta, \theta + \eps)}{\eps^\alpha} = f(0) \{1 + \beta \, r(\beta)\}.  \]

\subsubsection*{S2.2. Proof of Corollary~1}

For notational simplicity, write $\I=\I_n(\theta)$ and $D=D_\psi(\theta)$.  Then, for the case $\alpha=2$, it is easy to check that the lower bound in (9) from Theorem 1 is proportional to 
\[ \sup_{u: \|u\|=1} \frac{u^\top D^\top D u}{u^\top \I u}. \]
Let $\I = E L E^\top$ be the spectral decomposition of $\I$, and let $M=DEL^{-1/2}$; note that $M$ is $q \times d$ of rank $q \leq d$.  Then the above display equals 
\[ \sup_{v: \|v\|=1} v^\top M^\top M v = \lambda_{\max}(M^\top M). \]
The matrix $M^\top M$ is non-negative definite, in general, with $d-q$ eigenvalues equal to zero.  But the $q$ positive eigenvalues of $M^\top M$ coincide with those of the $q \times q$ positive definite matrix $M M^\top$.  And since 
\[ M M^\top = D E L^{-1/2} L^{-1/2} E^\top D^\top = D \I^{-1} D^\top, \]
it follows that $\lambda_{\max}(M^\top M) = \lambda_{\max}(MM^\top) = \lambda_{\max}(D \I^{-1} D^\top)$.

\subsubsection*{S2.3 \quad Proof of Proposition~3}

Let $\tilde h(\eta,\eta') = H^2(q_\eta, q_{\eta'})$, and recall that $\eta = g(\theta)$, a function of $\theta$.  Since $g$ is smooth, a small change in $\theta$ corresponds to a small change in $\eta$.  In particular, 
\[ g(\theta + \eps u) - g(\theta) = \eps \dot g(\theta)^\top u + o(\eps), \quad \eps \to 0. \]
Call the above difference $\delta$, so that $h(\theta; \theta + \eps u) = \tilde h(\eta, \eta + \delta)$ for small $\eps$.  Then 
\[ \frac{h(\theta; \theta + \eps u)}{|\eps|^\alpha} = \frac{\tilde h(\eta, \eta + \delta)}{|\eps|^\alpha} = \frac{\tilde h(\eta, \eta + \delta)}{|\delta|^\alpha} \, |\dot g(\theta)^\top u|^\alpha + o(1). \]
Now let $\eps \to 0$ and, therefore, $\delta \to 0$, and recall the definition of $\tilde J(\eta)$.  This proves that $J(\theta; u) = |\dot g(\theta)^\top u|^\alpha \tilde J(g(\theta))$ as was to be shown.

\subsubsection*{S2.4. Proof of Proposition~4}

Theorem~2 in the paper says that the optimal design must be symmetric.  So it suffices to show that the two-point symmetric design with points on the boundary, $\{(0.5,-A),(0.5,A)\}$, has information at least as big as the maximum among the symmetric designs.  That is, we intend to show that 
\begin{equation}\label{prop4:0}
     \J_{\{(0.5,-A),(0.5,A)\}}(\theta) \geq \max_{\xi \in \Xi_{\mathrm{sym}}} \J_{\xi}(\theta).
\end{equation}

Denote $\J_{\{(0.5, \pm A)\}}(\theta; u)$ as the Hellinger information of $\theta$ in the direction of $u$ based on design $\{(0.5,-A),(0.5,A)\}.$ For non-regular linear model, $g(\theta;x)=\theta_0+\theta_1x,$ the expressions for $\J_{\{(0.5, \pm A)\}}(\theta; u)$ and $\J_{\xi^\dagger}(\theta; u)$ for any symmetric design, $\xi^\dagger \in \Xi_{\mathrm{sym}}$, are listed below:
\begin{align*}
   \J_{\{(0.5, \pm A)\}}(\theta; u)&=0.5(|u_1+u_2A|^{\alpha}+|u_1-u_2A|^{\alpha}),\\ \J_{\xi^\dagger}(\theta; u)= &\sum_{i=1}^{m} w_i(|u_1+u_2x_i|^{\alpha}+|u_1-u_2x_i|^{\alpha}).
\end{align*}
Assuming that,
\begin{equation}\label{prop4:1}
    J_{\{(0.5, \pm A)\}}(\theta; u)\geq J_{\xi^\dagger}(\theta; u), \quad \text{for all unit vectors $u$ and all $\xi^\dagger\in\Xi_{\mathrm{sym}}$},
\end{equation}
let $\tilde u=\arg\min_{u:\|u\|=1} J_{(0.5,\pm A)}(\theta; u)$; then following from \eqref{prop4:1}, for any $\xi^\dagger,$
 \[ \underset{u:\|u\|_2=1}{\min}J_{(0.5,\pm A)}(\theta; u)=J_{(0.5,\pm A)}(\theta; \tilde u)\geq J_{\xi^\dagger}(\theta; \tilde u)\geq  \underset{u:\|u\|_2=1}{\min}J_{\xi^\dagger}(\theta; \tilde u).\]
Consequently, $\min_u J_{(0.5,\pm A)}(\theta; u) \geq \max_{\xi^\dagger} \min_u  J_{\xi^\dagger}(\theta; u),$ which implies \eqref{prop4:0}. 

To complete the proof, we only need to establish \eqref{prop4:1}.  Towards this, 
\begin{align*}
J_{\{(0.5, \pm A)\}}&(\theta; u)  -J_{\xi^\dagger}(\theta; u)\\
&= 0.5(|u_1+u_2A|^{\alpha}+|u_1-u_2A|^{\alpha})-\sum_{i=1}^{m} w_i(|u_1+u_2x_i|^{\alpha}+|u_1-u_2x_i|^{\alpha})\\
& = \sum_{i=1}^{m} w_i(|u_1+u_2A|^{\alpha}+|u_1-u_2A|^{\alpha}-|u_1+u_2x_i|^{\alpha}-|u_1-u_2x_i|^{\alpha})
\end{align*}
To evaluate the above expression, first see that one can rewrite any unit vector, except $u=(1,0)$,\footnote {The case for $u=(1,0)$ can be ignored, since $J_{\xi}(\theta; (1,0))$ are the same for all $\xi\in\Xi$.} in the following format: 
\begin{equation}\label{u-B}
u=\pm(\pm\frac{B}{\sqrt{1+B^2}},\frac{1}{\sqrt{1+B^2}}), \quad B\in[0,\infty).
\end{equation}
Notice that no matter what choices of sign combination of $u_1, u_2$ is given, 
\[|u_1+u_2x_i|^{\alpha}+|u_1-u_2x_i|^{\alpha}=(1+B^2)^{-0.5\,\alpha}(|B+x_i|^{\alpha}+|B-x_i|^{\alpha}),\]
then, 
\begin{align*}
&J_{\{(0.5, \pm A)\}}(\theta; u)-J_{\xi^\dagger}(\theta; u)&\\
=&(1+B^2)^{-0.5\,\alpha}\sum_{i=1}^{m} w_i(|B+A|^{\alpha}+|B-A|^{\alpha}-(|B+x_i|^{\alpha}+|B-x_i|^{\alpha})).
\end{align*}
Based on the expression above, to see if $J_{\{(0.5, \pm A)\}}(\theta; u)-J_{\xi^\dagger}(\theta; u)$ is non-negative or not for all $u,\xi^\dagger$ boils down to checking the sign of  $\sum_{i=1}^{m} w_i(|B+A|^{\alpha}+|B-A|^{\alpha}-(|B+x_i|^{\alpha}+|B-x_i|^{\alpha}))$ for any $0\leq x_1,..,x_m\leq A,B\in[0,\infty).$

The following shows that $|B+A|^{\alpha}+|B-A|^{\alpha}-(|B+x_i|^{\alpha}+|B-x_i|^{\alpha})$ is non-negative for all possible cases (three cases in total) defined by relationships between $x_i, A, B$ in location:  $ x_i\leq A\leq B $,  $ B\leq x_i\leq A$ and  $x_i\leq B\leq A $. 

\begin{itemize}
\item Case 1, $0 \leq x_i\leq A\leq B $, 
\[|B+A|^{\alpha}+|B-A|^{\alpha}-(|B+x_i|^{\alpha}+|B-x_i|^{\alpha})
=(B+A)^{\alpha}+(B-A)^{\alpha}-(B+x_i)^{\alpha}-(B-x_i)^{\alpha}\]

When $\alpha=1$,  
\[(B+A)+(B-A)-(B+x_i)-(B-x_i)=2B-B-B=0.\]
When $\alpha>1$, function $f_{1}(x)=(B+x)^{\alpha}+(B-x)^{\alpha}$ is an increasing function, since its first derivative is always positive when $B>x,$
\[\frac{\partial f_1(x)}{\partial x}=\alpha[(B+x)^{\alpha-1}-(B-x)^{\alpha-1}]>0.\]

Then, for any $x_i, 0\leq x_i\leq A$, $f_1(A)-f_1(x_i)\geq 0$, i.e.
\[(B+A)^{\alpha}+(B-A)^{\alpha}-(B+x_i)^{\alpha}-(B-x_i)^{\alpha}\geq0, \text{ for all } i=1,..,m.\]
\item Case 2: $0 \leq B\leq x_i\leq A$
\[|B+A|^{\alpha}+|B-A|^{\alpha}-(|B+x_i|^{\alpha}+|B-x_i|^{\alpha})
=(B+A)^{\alpha}+(A-B)^{\alpha}-(B+x_i)^{\alpha}-(x_i-B)^{\alpha}\]
When $\alpha=1$, 
\[(B+A)+(A-B)-(B+x_i)-(x_i-B)=2A-2x_i\geq 0.\]
When $\alpha>1$, function $f_2(x)=(B+x)^{\alpha}+(x-B)^{\alpha},  0\leq B<x$, is an increasing function since it's first derivative is always positive,  
\[\frac{\partial f_2(x)}{\partial x}=\alpha[(B+x)^{\alpha-1}+(x-B)^{\alpha-1}]>0 .\]

Since $x_i\leq A, f_2(A)-f_2(x_i)\geq 0$,  for all i, 

 \[(B+A)^{\alpha}+(A-B)^{\alpha}-(B+x_i)^{\alpha}-(x_i-B)^{\alpha}>0, i=1,...,m,\] 
\item Case 3 When $0 \leq x_i\leq B\leq A $,
\begin{align}
&|B+A|^{\alpha}+|B-A|^{\alpha}-(|B+x_i|^{\alpha}+|B-x_i|^{\alpha})\nonumber \\ 
=&(B+A)^{\alpha}+(A-B)^{\alpha}-(B+x_i)^{\alpha}-(B-x_i)^{\alpha}\nonumber \\ 
=&(B+A)^{\alpha}-(B+x_i)^{\alpha}+(A-B)^{\alpha}-(B-x_i)^{\alpha}.\label{Lem2-c3}
\end{align}

When $\alpha=1$, $(B+A)-(B+x_i)+(A-B)-(B-x_i)=2A-2B \geq 0$.

When $\alpha>1$, if $A-B\geq B-x_i\geq 0$, then $(A-B)^{\alpha}-(B-x_i)^{\alpha}\geq0
$, so (\ref{Lem2-c3}) is non-negative. 

When $\alpha>1$, if $0\leq A-B<B-x_i$, then $(A-B)^{\alpha}-(B-x_i)^{\alpha}<0$. Let 
$A-B=d_m, B-x_i=d_i$.
Notice that this assumption means $0\leq  d_m<d_i$. Set 
\[B+x_i=W, \text{ then }B+A=x_i+d_i+B+d_m=W+d_i+d_m.\]  
Consider $f_3(x)=(x+y)^{\alpha}-x^{\alpha}-y^{\alpha}, y>0, x\geq 0$, $f_3(x)$ is an increasing function, as its first derivative is positive, 
\[f_3'(x)=\alpha(x+y)^{\alpha-1}-\alpha x^{\alpha-1}>0.\]
Also notice that $f_3(0)=0$, so $f_3(x)$ is a non-negative function. 

Therefore, due to $w>0, d_i>0$, \[(W+d_i+d_m)^{\alpha}-(W)^{\alpha}-(d_i+d_m)^{\alpha}>0 \text{ and }(d_i+d_m)^{\alpha}-d_i^{\alpha}-d_m^{\alpha}>0.\]
Therefore, when $0\leq A-B<B-x_i$
\begin{align*}
&(B+A)^{\alpha}-(B+x_i)^{\alpha}+(A-B)^{\alpha}-(B-x_i)^{\alpha}\\
=&(W+d_i+d_m)^{\alpha}-(W)^{\alpha}+(d_m)^{\alpha}-(d_i)^{\alpha} \\
> &(d_i+d_m)^{\alpha}+ (d_m)^{\alpha}-(d_i)^{\alpha} \\
>&d_i^{\alpha}+d_m^{\alpha}+ (d_m)^{\alpha}-(d_i)^{\alpha}\\
\geq&0.
\end{align*}
\end{itemize}

In summary of all three cases, no matter where $B$ is in relation to $x_i$ and $A$, \[|B+A|^{\alpha}+|B-A|^{\alpha}-(|B+x_i|^{\alpha}+|B-x_i|^{\alpha})\geq0\text{ for all i=1,...,m}.\]

\subsection*{S2.5. Proof of Proposition~5}

Theorem~2 says that optimal design for the quadratic model must be a symmetric design, so here we only need to search among the collection of symmetric designs. 

Given any symmetric design 
\[ \xi^\dagger=\{(w_1,-x_1),...(w_m,-x_m),(w_1,x_1),...(w_m,x_m)\} \]
and direction vector $u$, for the non-regular quadratic regression model, the Hellinger information of $\xi^\dagger$ in the direction of $u=(u_1,u_2,u_3)$ has the expression of  
$J_{\xi^\dagger}(\theta; u)= \sum_{i=1}^{m} w_i(|u_1+u_2x_i+u_3x^2_i|^{\alpha}+ |u_1+u_2(-x_i)+u_3x^2_i|^{\alpha}). $ For simplicity, denote  $f_u(x)=u_1+u_2x+u_3x^2$;  then, when $\alpha=1$, the above becomes \[J_{\xi^\dagger}(\theta; u)= \sum_{1}^{m} w_i( |f_u(x_i)|+|f_u(-x_i)|).\]

Let's assume that there exists $r_i\in[0,1]$ such that, for all $x_i\in[-A,A]$, the following relation is true:
\begin{equation} \label{Obj.Quada1}
2r_i|f_u(0)|+(1-r_i)|f_u(A)|+(1-r_i)|f_u(-A)|>|f_u(x_i)|+|f_u(-x_i)|. 
\end{equation}
Then, given $w_i$, after multiplying $w_i$ on both sides of the inequality (\ref{Obj.Quada1}), 
we have 
\begin{equation}\label{Q1}
2w_{i}r_i|f(0)|+w_{i}(1-r_i)|f(A)|+w_{i}(1-r_i)|f(-A)|
\geq w_i(|f(x_i)|+ |f(-x_i)|).
\end{equation}

Let $w=\sum_{i=1}^{m}(1-r_i)w_i$, based on the fact that $\sum_{i=1}^{m}w_i=0.5, 1-2w= \sum_{i=1}^{m}2w_ir_i$. We can denote a three-point symmetric design based on the left hand side of (\ref{Q1}) as 
\[\xi_{w}=\{(w,-A),(1-2w,0),(w,A)\}, 0\leq w\leq 0.5.\] 
Hellinger information based on design $\xi_{w}$ in the direction of a given $u$ has the expression 
\[J_{\xi_{w}}(\theta; u)=(1-2w)|f(0)|+w|f(A)|+w|f(-A)|.\]
Thus, based on (\ref{Q1}), for any $u$, for any symmetric design $\xi^\dagger$, there is a $w$ such that  
\begin{equation*}
J_{\xi_{w}}(\theta; u)\geq J_{\xi^\dagger}(\theta; u).
\end{equation*}
Now, via Theorem~2 and the exact same argument that established Proposition~4, the conclusion of this proposition holds.

The only step we need in order to complete the proof is to show (\ref{Obj.Quada1}) is true. Notice that $|f_u(x)|=|f_{-u}(x)|$, i.e. $|u_1+u_2x+u_3x^2|=|-u_1-u_2x-u_3x^2|$. Thus, for every given $\bar u$ with $\bar u_3<0$, 
there is a $\dot u=-\bar u$ such that $|f_{\dot u}(x)|=|f_{\bar u}(x)|$, and  $f_{\dot u}(x)$ is convex. Thus, for simplicity, the following only shows (\ref{Obj.Quada1}) is true for $f_u(x)$ with $u_3>0$, i.e.  only when $f_u(x)$ is convex. There are seven cases based on the locations of x-intercepts of $f_u(x)$, and for each case,  (\ref{Obj.Quada1}) can be shown to be true. Here we only consider cases in which $u$ is such that its $u_3\neq0$, as the case for  $u_3=0$ is equivalent to the linear regression case. In the rest of the proof, for simplicity,  let $f(x)\equiv f_u(x)$. 

By convexity, if $f(x_i)>0$ over $[-B,B]$ for some $B>0$ and $x_i\in[0,B]$, and there is a $r_i\in(0,1)$, such that $x_i=r_i 0+(1-r_i)B$, and
\[r_if(0)+(1-r_i)f(B)>f(x_i)\text{ and } r_if(0)+(1-r_i)f(-B)>f(-x_i), \] then 
\begin{equation}\label{CV}
2r_i|f(0)|+(1-r_i)|f(B)|+(1-r_i)|f(-B)|>|f(x_i)|+|f(-x_i)|. 
\end{equation}

Given direction vector $u$ and design point location $-x_i,x_i$, with $x_i>0$ and the assumption that $u_3>0$, there are seven cases that describe the possible relationships between $-x_i,x_i$ and the left, right roots of $f(x)$, $x_L<x_R$.
\begin{itemize}
\item Case 1: $x_i<x_L,x_R$, 
\item Case 2: $x_L,x_R<-x_i$
\item Case 3: $-x_i\leq x_L,x_R\leq x_i$ 
\item Case 4: $x_L\leq -x_i , x_i\leq x_R$ 
\item Case 5: $-x_i\leq x_L \leq x_i\leq x_R$ 
\item Case 6: $x_L\leq -x_i \leq x_R\leq x_i$
\item Case 7: There is at most one root for $f(x)$, i.e. $f(x)\geq0$ for all $x\in R$
\end{itemize}

The following goes through these cases and shows that (\ref{Obj.Quada1}) is true for each of them. Notice that cases 1 and 2 are equivalent, and cases 5 and 6 are equivalent. So we shall focus on cases 1, 3, 4, 5, and 7.

\begin{itemize}
\item  In case 1 both roots are above $x_i$; there are two possible ways that this can happen regarding the given value of A: 
\begin{itemize}
\item 1.1) The left root $x_L$ is above A, i.e. $ A\leq x_L$. This implies that $f(x_i)>0$ over $[-A,A]$, so by the argument of convexity in (\ref{CV}), (\ref{Obj.Quada1}) is true.

\item 1.2) The left root $x_L$ is below A, i.e. $ x_L< A$. Here, $f(-x_i),f(x_i),f(-A)>0$, which implies that 
\begin{equation}\label{Case1.2.1}
|f(-x_i)|+|f(x_i)|=2u_1+2u_3x^2, \text{ and } f(-A)=u_1-u_2A+u_3A^2.
\end{equation}
If $A$ is smaller than right root, $A<x_R$, then $f(A)<0$, so

$|f(A)|=-u_1-u_2A-u_3A^2>0$, and $-u_2A>u_1+u_3A^2$. Then with (\ref{Case1.2.1}),  
\[|f(A)|+|f(-A)|=-2u_2A>2u_1+2u_3A^2>2u_1+2u_3x_i^2=|f(-x_i)|+|f(x_i)|.\]

If $A$ is larger than right root, $A>x_R$, then $f(A)>0$, so

$|f(A)|=u_1+u_2A+u_3A^2>0$. Then with (\ref{Case1.2.1}),
\[|f(A)|+|f(-A)|=2u_1+2u_3A^2>2u_1+2u_3x_i^2=|f(-x_i)|+|f(x_i)|.\]
Then for 1.2) one can find a ratio $r_A$ such that $r_A(|f(A)|+f(-A))>f(x_i)+f(-x_i)$,  letting $r_i=1-r_A$, then (\ref{Obj.Quada1}) is true, i.e. 
\[2r_i|f(0)|+(1-r_i)|f(A)|+(1-r_i)|f(-A)|>|f(x_i)|+|f(-x_i)|. \]
\end{itemize}

\item Case 3: $-x_i\leq x_L,x_R\leq x_i$, is the case of both roots of $f(x)$ are in $[-x_i,x_i]$, so $f(x)$ would be positive and increasing over $[x_i,A]$, while positive and decreasing over $[-A,-x_i]$, i.e. \[f(A)>f(x_i)>0,\quad f(-A)>f(-x_i)>0,\] 


Let $r_i=1-r_A$, then, under $\alpha=1$,  (\ref{Obj.Quada1}) is true, i.e., 
\[2r_i|f(0)|+(1-r_i)|f(A)|+(1-r_i)|f(-A)|>|f(x_i)|+|f(-x_i)|. \]
\item Case 4: $x_L\leq -x_i , x_i\leq x_R$. In this case,  $f(x)\leq 0$ over $[-x_i,x_i]$, which means $|f(x)|=-f(x)=-u_1-u_2x-u_3x^2$ is concave over  $[-x_i,x_i]$. Thus, $|f(0)|>\frac{1}{2}|f(x_i)|+\frac{1}{2}|f(-x_i)|$, and consequently,  (\ref{Obj.Quada1}) holds. 

\item Case 5: $-x_i\leq x_L \leq x_i\leq x_R$.

First, the assumption of case 5, $-x_i\leq x_L \leq x_i\leq x_R$, implies that $\frac{-u_2}{2u_3}=\frac{x_L+x_R}{2}>\frac{-x_i+x_i}{2}=0$, i.e. $u_2<0$. 

Also notice that $-x_i\leq x_L$ implies that $0<f(-x_i)<f(-A)$ and 
\begin{equation}\label{case5.c1}
|f(-A)|=u_1-u_2A+u_3A^2,\quad |f(-x_i)|=u_1-u_2x_i+u_3x_i^2.
\end{equation}
Based on the set up of case 5, and the possible relations of $A$ and direction $u$, the expression of $f(x_i)$ and $f(A)$ depends on the following two sub-cases: 

\begin{itemize}
\item The right boundary $A$ is below right intercept, i.e. $A<x_R$, i.e.
$f(x_i)<0$ and $f(A)<0$, so
\[|f(A)|=-f(A)=-u_1-u_2A-u_3A^2,\text{ and } |f(x_i)|=-f(x_i)=-u_1-u_2x_i-u_3x_i^2.\]
Therefore, with the fact that $-u_2> 0$, $A\geq x_i$, and (\ref{case5.c1}), we have 
\begin{align*}
&|f(-A)|+|f(A)|-|f(-x_i)|-|f(x_i)|\\
=&-2u_2A+2u_2x_i\\
=&-2u_2(A-x_i)\\
\geq& 0.
\end{align*}

\item The right boundary $A$ is above right intercept, i.e.  $x_R<A$, 
 which implies that $f(x_i)<0<f(A)$, 
\[|f(A)|=u_1+u_2A+u_3A^2,\text{ and } |f(x_i)|=-f(x_i)=-u_1-u_2x_i-u_3x_i^2.\]
Therefore, with $-u_2>0$, $A\geq x_i$, and (\ref{case5.c1})  
\begin{align*}
&|f(-A)|+|f(A)|-|f(-x_i)|-|f(x_i)|\\
=&2u_1+2u_3A^2+2u_2x_i\\
=&2(u_1+u_2A+u_3A^2)-u_2(A-x_i)\\
=&2|f(A)|-u_2(A-x_i)\\
\geq&0.
\end{align*}

\end{itemize}
Combining these two sub-cases, we can conclude that under case 5, 
\[|f(-A)|+|f(A)|\geq |f(-x_i)|+|f(x_i)|.\]
Then one can find a ratio $r_A$ such that $r_A(|f(A)|+f(-A))>f(x_i)+f(-x_i)$,  letting $r_i=1-r_A$, then (\ref{Obj.Quada1}) is true, as 
\[2r_i|f(0)|+(1-r_i)|f(A)|+(1-r_i)|f(-A)|>|f(x_i)|+|f(-x_i)|. \]
\item Case 7: There is at most one root, which means, $f(x)\geq0$ for all $x\in [-A,A]$. Thus, by the argument in (\ref{CV}), implies (\ref{Obj.Quada1}).


\end{itemize}
In summary of these all seven cases, (\ref{Obj.Quada1}) holds.

\subsubsection*{S2.6. Proof of Lemma~1}

Define the mean function of the estimator $T$, i.e., $m_\psi(\theta) = E_\theta(T)$.  Since integration of a constant function with respect to the (signed) measure with density $p_\theta - p_\vartheta$ is zero, we have the following identity:
\[ m_\psi(\theta) - m_\psi(\vartheta) = \int \bigl[ T(y) - \tfrac12\{m_\psi(\theta) + m_\psi(\vartheta)\} \bigr] \bigl[ p_\theta(y) - p_\vartheta(y) \bigr] \, \mu(dy). \]
Write $v_{\theta,\vartheta}(y)=T(y) - \tfrac12\{m_\psi(\theta) + m_\psi(\vartheta)\}$.  Now bound the norm of the quantity in the above display:
\begin{align*}
\|m_\psi(\theta)-m_\psi(\vartheta)\| & = \Bigl\| \int v_{\theta,\vartheta} (p_\theta - p_{\vartheta}) \,dy \Bigr\| \\
& \leq \int \|v_{\theta,\vartheta}\| \, |p_\theta^{1/2} + p_{\vartheta}^{1/2}| \, | p_\theta^{1/2}-p_{\vartheta}^{1/2}| \,dy.
\end{align*}
Next, apply the Cauchy--Schwartz inequality, to get 
\[ \|m_\psi(\theta)-m_\psi(\vartheta)\|^2 \leq \int \|v_{\theta,\vartheta}\|^2 \, |p_\theta^{1/2} + p_{\vartheta}^{1/2}|^2 \,dy \cdot h(\theta;\vartheta). \]
For two non-negative numbers $a$ and $b$, we have $(a+b)^2 \leq 2(a^2 + b^2)$, so the first term in the above upper bound is itself bounded by 
\[ 2 \int \|v_{\theta,\vartheta}\|^2 p_\theta \,dy + 2 \int \|v_{\theta,\vartheta}\|^2 p_{\vartheta} \,dy. \]
If we rewrite $v_{\theta,\vartheta}$ as 
\[ v_{\theta, \vartheta}(y) = \{T(y) - m_\psi(\theta)\} + \tfrac12\{ m_\psi(\vartheta) - m_\psi(\theta)\}, \]
and use the fact that $\int \{T - m_\psi(\theta)\} p_\theta \,dy = 0$, then we get 
\[ \int \|v_{\theta,\vartheta}\|^2 \, p_\theta \,dy \leq R_\psi(T, \theta) + \tfrac14\|m_\psi(\theta) - m_\psi(\vartheta)\|^2. \]
An analogous bound holds for $\int \|v_{\theta,\vartheta}\|^2 \, p_\vartheta \,dy$, yielding the expression 
\[ \|m_\psi(\theta)-m_\psi(\vartheta)\|^2 \leq 2 h(\theta; \vartheta) \bigl\{R_\psi(T,\theta) + R_\psi(T,\vartheta) + \tfrac12\|m_\psi(\theta) - m_\psi(\vartheta)\|^2 \bigr\}. \]

Rearranging terms gives the bound 
\[ R_\psi(T,\theta) + R_\psi(T,\vartheta) \geq \frac{1-h(\theta;\vartheta)}{2h(\theta; \vartheta)} \|m_\psi(\theta) - m_\psi(\vartheta)\|^2. \]
Finally, write $b_\psi(\theta) = m_\psi(\theta) - \psi(\theta)$ for the bias function of $T$, and consider the following two exhaustive cases based on the magnitude of the bias:
\begin{itemize}
\item Suppose that $\max\{|b_\psi(\theta)|, |b_\psi(\vartheta)|\} < \frac14 \|\psi(\theta) - \psi(\vartheta)\|$.  Then it follows from the triangle inequality that 
\[ \|m_\psi(\theta) - m_\psi(\vartheta)\| = \|\psi(\theta) - \psi(\vartheta) + b_\psi(\theta) - b_\psi(\vartheta)\| \geq \tfrac12 \| \psi(\theta) - \psi(\vartheta)\|. \]
\item Next, suppose that, say, $\|b_\psi(\theta)\| \geq \frac14\|\psi(\theta) - \psi(\vartheta)\|$.  Then we trivially have $R_\psi(T,\theta) \geq \|b_\psi(\theta)\|^2$ and, therefore, $R_\psi(T, \theta) + R_\psi(T, \vartheta) \geq \tfrac{1}{16} \|\psi(\theta) - \psi(\vartheta)\|^2$.  
\end{itemize}
Putting these two cases together proves the claim.


\end{document}